\documentclass[preprint,12pt]{elsarticle}

\usepackage{amsthm}
\usepackage[english]{babel}
\usepackage{mathtools}
\usepackage{tikz}
\usetikzlibrary{positioning,calc} 
\usetikzlibrary{arrows.meta, shapes.geometric} 
\usetikzlibrary{automata, arrows}
\usepackage[normalem]{ulem}
\usetikzlibrary{fit}  
\usepackage{tikz-cd}
\tikzset{commutative diagrams/.cd}
\tikzcdset{row sep/tiny=0.12 em}
\usepackage{verbatim}
\usepackage{enumitem}
\usepackage{amsfonts}
\usepackage[parfill]{parskip}
\usepackage{amsmath,amsthm,amssymb,epic,eepic,bm}
\usepackage{eqlist,eqparbox}
\usepackage{comment}
\usepackage{caption}
\theoremstyle{plain}
\newtheorem{theorem}{Theorem}[section]
\newtheorem{lemma}[theorem]{Lemma}
\newtheorem{proposition}[theorem]{Proposition}

\theoremstyle{definition}
\newtheorem{definition}[theorem]{Definition}

\newtheorem{remark}[theorem]{Remark}
\newtheorem{example}[theorem]{Example}

\DeclareRobustCommand{\mujmsf}[1]{%
  \ifcat\noexpand#1\relax\mujmsfgreek{#1}\else\mathsf{#1}\fi
}
\makeatletter
\newcommand{\mujmsfgreek}[1]{\csname s\expandafter\@gobble\string#1\endcsname}
\makeatother

\DeclareSymbolFont{sfletters}{OML}{cmbrm}{m}{it}

\DeclareMathSymbol{\salpha}{\mathord}{sfletters}{"0B}
\DeclareMathSymbol{\sbeta}{\mathord}{sfletters}{"0C}
\DeclareMathSymbol{\sgamma}{\mathord}{sfletters}{"0D}
\DeclareMathSymbol{\sdelta}{\mathord}{sfletters}{"0E}
\DeclareMathSymbol{\sepsilon}{\mathord}{sfletters}{"0F}
\DeclareMathSymbol{\szeta}{\mathord}{sfletters}{"10}
\DeclareMathSymbol{\seta}{\mathord}{sfletters}{"11}
\DeclareMathSymbol{\stheta}{\mathord}{sfletters}{"12}
\DeclareMathSymbol{\siota}{\mathord}{sfletters}{"13}
\DeclareMathSymbol{\skappa}{\mathord}{sfletters}{"14}
\DeclareMathSymbol{\slambda}{\mathord}{sfletters}{"15}
\DeclareMathSymbol{\smu}{\mathord}{sfletters}{"16}
\DeclareMathSymbol{\snu}{\mathord}{sfletters}{"17}
\DeclareMathSymbol{\sxi}{\mathord}{sfletters}{"18}
\DeclareMathSymbol{\spi}{\mathord}{sfletters}{"19}
\DeclareMathSymbol{\srho}{\mathord}{sfletters}{"1A}
\DeclareMathSymbol{\ssigma}{\mathord}{sfletters}{"1B}
\DeclareMathSymbol{\stau}{\mathord}{sfletters}{"1C}
\DeclareMathSymbol{\supsilon}{\mathord}{sfletters}{"1D}
\DeclareMathSymbol{\sphi}{\mathord}{sfletters}{"1E}
\DeclareMathSymbol{\schi}{\mathord}{sfletters}{"1F}
\DeclareMathSymbol{\spsi}{\mathord}{sfletters}{"20}
\DeclareMathSymbol{\somega}{\mathord}{sfletters}{"21}
\DeclareMathSymbol{\svarepsilon}{\mathord}{sfletters}{"22}
\DeclareMathSymbol{\svartheta}{\mathord}{sfletters}{"23}
\DeclareMathSymbol{\svarpi}{\mathord}{sfletters}{"24}
\DeclareMathSymbol{\svarrho}{\mathord}{sfletters}{"25}
\DeclareMathSymbol{\svarsigma}{\mathord}{sfletters}{"26}
\DeclareMathSymbol{\svarphi}{\mathord}{sfletters}{"27}

\newcommand{\leqs}{\leqslant}

\newcommand{\FSUPH}{{\mbox{$F\kern -.85ex-\kern -.35ex\mathbb{S}$}}}
\newcommand{\FSUPL}{{\mbox{$F\kern -.85ex-\kern -.35ex\mathbb{S}_\leqs$}}}
\newcommand{\smFSUPL}{{{F\kern -.35ex-\kern -.08ex\mathbb{S}_\leqs}}}
\newcommand{\FINFH}{{\mbox{$F\kern -.85ex-\kern -.35ex\mathbb{M}$}}}
\newcommand{\FINFL}{{\mbox{$F\kern -.85ex-\kern -.35ex\mathbb{M}_\leqs$}}}
\newcommand{\VSUPH}{{\mbox{${\mathbf V}\kern -.85ex-\kern -.35ex\mathbb{S}$}}}
\newcommand{\VJ}{{\mbox{${\mathbf V}\kern -.85ex-\kern -.35ex\mathbb{J}$}}}
\newcommand{\VMod}{{\mbox{$V\kern -.85ex-\kern -.35ex\mathbf{Mod}$}}}
\newcommand{\VFSUPH}{{\mbox{${\mathbf V}\kern -.85ex-\kern -.35ex {\textnormal F}\kern -.85ex-\kern -.35ex\mathbb{S}$}}}
\newcommand{\VFSUPL}{{\mbox{${\mathbf V}\kern -.85ex-\kern -.35ex {\textnormal F}\kern -.85ex-\kern -.35ex\mathbb{S}_\leqs$}}}
\newcommand{\VSUPLF}{{\mbox{${\mathbf V}\kern -.85ex-\kern -.35ex F\kern -.85ex-\kern -.35ex\mathbb{S}_\leqs$}}}
\newcommand{\dvaSUPH}{{\mbox{${\mujmsf{2}}\kern -.85ex-\kern -.35ex\mathbb{S}$}}}
\newcommand{\indexVFSUPH}{{{\mathbf V}\kern -.485ex-\kern -.1735ex {\textnormal F}\kern -.485ex-\kern -.1735ex\mathbb{S}}}
\newcommand{\indexVFSUPL}{{{\mathbf V}\kern -.485ex-\kern -.1735ex {\textnormal F}\kern -.485ex-\kern -.1735ex\mathbb{S}_\leqs}}
\makeatletter
\newcommand{\firstarrow}[2][default]{
  \expandafter\gdef\csname arrow@content@#1\endcsname{#2}
  \xrightarrow{\@nameuse{arrow@content@#1}}
}
\newcommand{\alignarrow}[1][default]{\xrightarrow{\hphantom{\@nameuse{arrow@content@#1}}}}
\makeatother

\newcommand{\homs}[1]{\textnormal{Hom}_{\mathbf V}{#1}}
\newcommand{\Vmodule}{${\mathbf V}$-module{}}
\newcommand{\Vsubmodule}{${\mathbf V}$-submodule{}}
\newcommand{\Vrelation}{${\mathbf V}$-relation{}}
\newcommand{\Vframe}{${\mathbf V}$-frame{}}
\newcommand{\VFsemilattice}{{$\mathbf V-\textnormal{F}$}-semilattice}
\newcommand{\VFcongruence}{{$\mathbf V-\textnormal{F}$}-congruence}
\newcommand{\VFprenucleus}{{$\mathbf V-\textnormal{F}$}-prenucleus}
\newcommand{\VFnucleus}{{$\mathbf V-\textnormal{F}$}-nucleus}
\newcommand{\VFnuclei}{{$\mathbf V-\textnormal{F}$}-nuclei}
\newcommand{\VFprenuclei}{{$\mathbf V-\textnormal{F}$}-prenuclei}

\newcommand{\arw}[1]{{\firstarrow{#1}}}

\newcommand{\adjunction}[4]{(#1, #2) \colon #3 \dashv #4 \colon}
\newcommand{\adjpair}{\rightleftarrows}
\newcommand{\adjpairop}{\leftleftarrows}

\begin{document}

\begin{frontmatter}
\title{Many-valued aspects of tense an related operators}


\author[label2]{Michal Botur}
	\author[label1]{Jan Paseka} \author[label1]{Richard Smolka}
	


\affiliation[label1]{organization={Department of Mathematics and Statistics}, 
            addressline={Faculty of Science, Masaryk University, Kotlářská 2}, 
            city={Brno},
            postcode={CZ-611 37}, 
            country={Czech Republic}}

\affiliation[label2]{organization={Department of Algebra and Geometry}, 
            addressline={Faculty of Science, Palacký University Olomouc, 17.\ listopadu 12}, 
            city={Olomouc},
            postcode={CZ-771 46}, 
            country={Czech Republic}}            

\begin{abstract}
Our research builds upon Halmos's foundational work on functional monadic Boolean algebras and our previous work on tense operators to develop three essential  constructions, including the important concepts of fuzzy sets and powerset operators. These constructions have widespread applications across contemporary mathematical disciplines, including algebra, logic, and topology. The framework we present generates four covariant and two contravariant functors, establishing three adjoint situations.
\end{abstract}

\begin{graphicalabstract}

\vskip3cm

\begin{center}
\resizebox{\textwidth}{!}{%
\begin{tikzpicture}[->,>=stealth',shorten >=1pt,auto,thick,
  main node/.style={circle,draw,font=\sffamily\Large\bfseries}]
  \node[main node,fill=green,font=\normalsize] (s1) at (0,0) {\begin{tabular}{c}
  \bf Quantale\\ \bf modules
  \end{tabular}};
  \node[main node,fill=yellow,font=\normalsize] (s2) at (8,0) {\begin{tabular}{c}
  \bf Tense fuzzy\\ \bf algebras
  \end{tabular}};
  \node[main node,,fill=brown,font=\normalsize] (s3) at (4,-3) {\begin{tabular}{c}
  \bf Fuzzy\\ \bf relations
  \end{tabular}};

  \draw[->, blue!50] (s1) to[bend left] (s2);
  \draw[->, blue!50] (s1) to[bend right] (s3);
  \draw[->, blue!50] (s2) to[bend left=12] (s1);
  \draw[->, red!50]  (s2) to[bend right=12] (s3);
  \draw[->, blue!50] (s3) to[bend right=12] (s1);
  \draw[->, red!50]  (s3) to[bend right] (s2);
  
  \node at (4,2) {\textcolor{blue}{$(-)^{\mathbf J}$}};          
  \node at (3.2, -1.2) {\textcolor{blue}{\phantom{i}$\mathbf J\otimes (-)$}};     
  \node at (4, -0.2) {\textcolor{blue}{$(-)\otimes\mathbf H$}};      
  \node at (7.75, -2.3) {\textcolor{red}{\phantom{ii}$\mathbf J[(-), \mathbf L]$}}; 
  \node at (0.23, -2.3) {\textcolor{blue}{$\mathbf J[\mathbf H, (-)]$\phantom{ii}}}; 
  \node at (5.19, -1.2) {\textcolor{red}{${\mathbf L}^{(-)}$\phantom{iii}}};          
\end{tikzpicture}}
\end{center}

\end{graphicalabstract}



\end{frontmatter}

	\section{Introduction}\label{sec:intro} 

Fuzzy logic is a form of many-valued logic that extends beyond traditional binary (true/false) logic. Some situations may benefit from a wider range of possibilities, allowing for statements to be partially true or false to varying degrees. Fuzzy logic uses a complete lattice of truth values, where $0$ represents completely false and $1$ represents completely true, with various levels in between.


 The next step is to extend this complete lattice with additional operations, including (some or all) logical operations such as conjunctions, implications (in the case of non-commutative logics, with two implications), negations, or others. We consider the fundamental structure for all theories to be a complete residuated lattice \cite{gajiko}, which is an  algebra ${\mathbf V}=(V,\vee,\wedge,\cdot,/,\backslash, e)$ of the type $\langle 2,2,2,2,2,0\rangle$ where $(V,\vee,\wedge)$ is a complete lattice, $(V,\cdot,e)$ is a monoid and the adjoint property $$x\cdot y\leq z \text{ if and only if }x\leq z/y \text{ if and only if }y\leq x\backslash z$$ holds.

 It is, of course, well known that residua operations, if they exist, are determined uniquely by 
 $$
 \begin{aligned}
     x\backslash y &= \max\{z\in L\mid x\cdot z\leq y\}\\
     y/ x &= \max\{z\in L\mid z\cdot x\leq y\}.
 \end{aligned}
$$
Moreover, residual operations $/$ and $\backslash$ exist if and only if the following conditions hold for all $x,y_i\in V; i\in I:$ 
 $$
 \begin{aligned}
x\cdot \bigvee_{i\in I}y_i &= \bigvee_{i\in I}(x\cdot y_i)\\
(\bigvee_{i\in I}y_i)\cdot x &= \bigvee_{i\in I}( y_i\cdot x).
  \end{aligned}
$$
 
 In addition to the well-known fact that the infimum operation is determined by the supremum operation by
 $$\bigwedge_{i\in I}x_i=\bigvee\{x\in L\mid  (\forall i\in I)(x\leq x_i)\}$$
 we can define a residuated lattice using the structure of an algebra 
 ${\mathbf V}=(V,\bigvee, \otimes, e)$ where $(V,\bigvee)$ is a complete join semilattice (shortly $\bigvee$-semilattice), $(V,\otimes,e)$ is a monoid and the following equalities hold: 
 $$
 \begin{aligned}
x\otimes \bigvee_{i\in I}y_i &= \bigvee_{i\in I}(x\otimes y_i)\\
(\bigvee_{i\in I}y_i)\otimes x &= \bigvee_{i\in I}( y_i\otimes x).
  \end{aligned}
$$

These algebras, known as {\em unital quantales} \cite{KP}, are already established in theoretical mathematics. They have origins in pointfree topology (also known as locale theory) and have applications in quantum logic. Unital quantales provide a generalized algebraic framework that connects these seemingly disparate areas of mathematics, 
theoretical computer science and physics.

This paper studies special kinds of morphisms and operators inspired by tense operators, which are a special case of modal operators. Tense operators were first introduced on Boolean algebras to add tenses to the language of standard propositional logic (see for example \cite{Kowalski}). The idea of these operators is based on the so-called time frame $(T,\rho)$, which is the set of "instants" together with the relation $\rho$ expressing the temporal sequence. We interpret the expression $i\rho j$ 
meaning that instant $i$ precedes instant $j$ or, dually, that instant $j$ 
follows instant $i$.

 If $\mathbf L$ is a complete lattice representing a scale of truth values (e.g., in the case of Boolean algebras a two-element lattice, in the case of MV-algebras an interval $[0,1]$ of real numbers, etc.) then ${L}^T$ represents a set of propositions with a temporal dimension. We interpret the element $x\in L^T$ as a proposition 
 whose truth value at instant $i\in T$ is $x(i)$.

Given this temporal dimension, we can define operators $F$ and $P$ on the set $L^T$ having the following meaning: 
\begin{itemize}
    \item[-] $Fx$ means "(sometimes) $x$ will be true",
    \item[-] $Px$ means "the statement $x$ was (sometimes) true".
\end{itemize} 
We formally define these operators by
$$(Fx)(i)=\bigwedge_{i\rho k}x(k),\,\, (Px)(j)=\bigwedge_{l\rho j}x(l),$$
where $i$ and $j$ represent time instants.

The operators $G$ and $H$ are dually defined on the set $L^T$ such that: 
\begin{itemize}
    \item[-] $Gx$ means "the statement $x$ will always be true in the future,"
    \item[-] $Hx$ means "the statement $x$ has always been true in the past". 
\end{itemize}
 Formally, we introduce these operators by
 $$(Fx)(i)=\bigvee_{i\rho k}x(k),\,\, (Gx)(j)=\bigvee_{l\rho j}x(l).$$
 These operators have been extensively studied from various perspectives, including algebraic, logical, and computational. 
For Boolean algebras, see \cite{Kowalski}, for MV-algebras, see \cite{Georgescu}, \cite{BotPa}, and others \cite{FiPeSa}.

The final inspiration for our main concept and the focus of this paper comes from the fuzzification of binary relations, and consequently, of time frames, as studied  in  \cite{BotM}. In this context,  a fuzzy binary relation on  set $T$ 
 is represented as  $r\colon T^2\longrightarrow L$, where $r(i,j)$ expresses the truth value of the statement "the instant $i$ precedes the instant  $j$". 
 
 For a unital quantale  $(V, \bigvee, \otimes,e)$  we can 
 define an operator $F\colon V^T\longrightarrow V^T$ by $$(Fx)(i)=\bigvee_{i\in T}r(i,j)\otimes x(j).$$ 

 The same procedure applies if we substitute ${\mathbf V}$ with a module structure 
 over ${\mathbf V}$.

It's worth noting that this definition reduces to the earlier one in the case of a boolean binary relation.

These types of mappings have numerous interpretations in non-classical logic theory and serve as important tools. For further details on these operators, we direct the reader to works such as \cite{capebo,chapa,goje}.

In this paper, we aim to demonstrate that by translating this theory into the language of quantales, we uncover a rich functorial background. This translation not only clarifies the existing theory but also provides deeper insights.

The structure of this paper is as follows. 
After this introduction, Section \ref{Preliminaries} introduces the necessary theory to extend the fundamental constructions from \cite{Bot} to arbitrary unital quantales, broadening their applicability beyond the specific case of a trivial, 2-element quantale. We will define a generalized concept of a time frame $(T,\rho)$, namely a \Vframe{} $(T,r)$, and prove essential technical lemmas. 

Section \ref{Basic constructions and functorial theorems} 
establishes the core principles and structures that will underpin our subsequent research. We begin with a key assumption: $\mathbf V$ is taken to be a commutative unital quantale throughout our work.

First, we derive three fundamental constructions.

Namely, we construct: 
\begin{enumerate}[label=\textup{(\roman*)}]
    \item a \Vmodule\ $\mathbf{L}^{\mathbf{J}}$ with a unary operation 
    $F$ (called shortly an \emph{\VFsemilattice}) from a \Vmodule\ $\mathbf{L}$
    and a fuzzy relation $\mathbf{J}$ (called a \emph{\Vframe}), 
    
    \item a \Vmodule\  $\mathbf{J} \otimes \mathbf{H}$ from a {\Vframe} $\mathbf{J}$ and 
    a {\VFsemilattice} $\mathbf{H}$, and  
    
    \item a {\Vframe} $\mathbf{J}[\mathbf{H}, \mathbf{L}]$ from a {\VFsemilattice} $\mathbf{H}$ 
    and a \Vmodule\   $\mathbf{L}$. 
\end{enumerate}

These constructions will play a pivotal role in our forthcoming discussions and analyses.

Visual representation of this situation is depicted in the picture below, where blue lines correspond to covariant functors and red lines correspond to contravariant functors: 

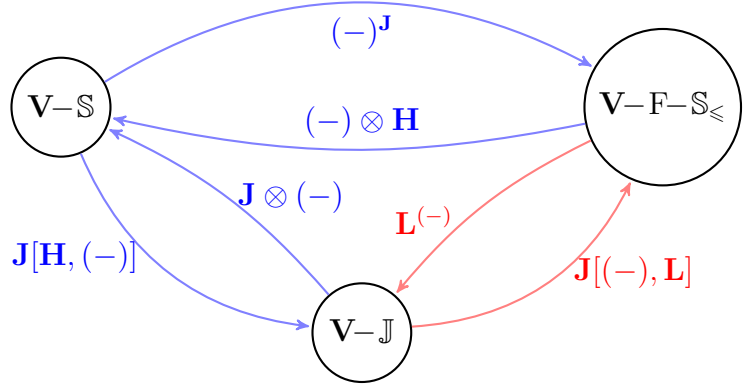
\begin{center}
\begin{tikzpicture}[->,>=stealth',shorten >=1pt,auto,thick,
  main node/.style={circle,draw,font=\sffamily\Large\bfseries}]
  \node[main node,font=\normalsize] (s1) at (0,0) {\VSUPH};
  \node[main node,font=\normalsize] (s2) at (8,0) {\VFSUPL};
  \node[main node,font=\normalsize] (s3) at (4,-3) {\VJ};

  \draw[->, blue!50] (s1) to[bend left] (s2);
  \draw[->, blue!50] (s1) to[bend right] (s3);
  \draw[->, blue!50] (s2) to[bend left=12] (s1);
  \draw[->, red!50]  (s2) to[bend right=12] (s3);
  \draw[->, blue!50] (s3) to[bend right=12] (s1);
  \draw[->, red!50]  (s3) to[bend right] (s2);
  
  \node at (4,1) {\textcolor{blue}{$(-)^{\mathbf J}$}};          
  \node at (3, -1.2) {\textcolor{blue}{\phantom{i}$\mathbf J\otimes (-)$}};     
  \node at (4, -0.2) {\textcolor{blue}{$(-)\otimes\mathbf H$}};      
  \node at (7.5, -2.2) {\textcolor{red}{\phantom{ii}$\mathbf J[(-), \mathbf L]$}}; 
  \node at (0.3, -2.0) {\textcolor{blue}{$\mathbf J[\mathbf H, (-)]$\phantom{ii}}}; 
  \node at (5, -1.5) {\textcolor{red}{${\mathbf L}^{(-)}$\phantom{iii}}};          
\end{tikzpicture}
\end{center}

Here \VSUPH\ is the category of \Vmodule{}s, 
{\VJ} is the category of {\Vframe}s, and 
{\VFSUPL} the category of {\VFsemilattice}s.

Extending the results of \cite[Section 4]{Bot}, we establish three adjoint situations tailored to quantale-valued structures:
\begin{enumerate}
    \item[(i)] $\adjunction{\eta}{\varepsilon}{\mathbf{J} \otimes -}{(-)^{\mathbf{J}}}$ between categories:
    \[
    \VSUPH \adjpair \VFSUPL
    \]
    
    \item[(ii)] $\adjunction{\varphi}{\psi}{- \otimes \mathbf{H}}{\mathbf{J}[\mathbf{H}, -]}$ forming the correspondence:
    \[
    \VSUPH \adjpair \VJ
    \]
    
    \item[(iii)] $\adjunction{\nu}{\mu}{\mathbf{J}[-, \mathbf{L}]}{\mathbf{L}^{-}}$ inducing the duality:
    \[
    \VJ \adjpairop \VFSUPL
    \]
\end{enumerate}

To concretize some of these adjoint pairs, Section~\ref{Examples} develops a  prototypical example. Finally, Section~\ref{conclusion} synthesizes our key findings and discusses their categorical implications.

For unexplained concepts and terminology, please consult references \cite{Joy of cats}, \cite{KP}, and \cite{ZL3}. While we have strived to make this paper self-contained, space limitations have necessitated omitting certain details, which readers will need to explore independently.

	\section{Preliminaries}\label{Preliminaries}
	
In this section, we will introduce all the necessary theory which we will need to extend the fundamental constructions we have used in \cite{Bot} to accommodate arbitrary unital quantales, broadening their applicability beyond a specific case of a trivial, 2-element quantale. The notion of quantales and \Vmodule{}s \cite{KP} play a crucial role in quantum logic, where quantales serve as algebraic models for reasoning about propositions with continuous truth values. 

By providing a mathematical framework for generalizing classical logical operations to systems with infinite or continuous states, these constructions facilitate the study of quantum phenomena. Overall, while quantales enable the modeling of continuous truth values inherent in quantum propositions, \Vmodule{}s formalize the linear structure of quantum states and operators. 

Moreover, the connection between quantale modules and fuzzy logic provides a rich mathematical framework for studying and developing fuzzy logical systems. It allows for the application of algebraic and category-theoretic methods to fuzzy logic, potentially leading to new insights and applications in areas such as artificial intelligence, decision theory, and control theory.

\subsection{Notation}

 We will use the standard category-theoretic notation for composition of maps, that is, 
for maps $f\colon A\longrightarrow B$ and $g\colon B\longrightarrow C$
we denote their composition by $g\circ f\colon A\longrightarrow C$, so that
$(g\circ f)(a) = g(f(a))$ for all $a\in A$. The set of all
maps from the $A$ to $B$ we denote by the usual $B^A$. For a map
$f\colon A\longrightarrow B$ and a set $I$ we write
$f^I\colon A^I\to B^I$ for the map defined by $f^I(x)(i)=f(x(i))$ evaluated on $i \in I$.\\

The algebraic foundation for our work rests on specific structures, which we now introduce. We begin by recalling the definition of a \emph{unital quantale}: a complete lattice with an associative multiplication distributing over arbitrary joins and a multiplicative unit. This concept is fundamental in both nonclassical logics and theoretical computer science.

	\begin{definition}\label{keyquantale}
	A {\em unital quantale} is a 4-tuple $\mathbf V=(V,\bigvee, \otimes,e)$, where $(V,\bigvee)$ is a $\bigvee$-semilattice, $e\in V$ 
	and $\otimes$ is a binary operation on $V$ satisfying:
	\begin{itemize}
		\item[](V1) $a \otimes (b \otimes c) = (a \otimes b) \otimes c$ for all $a,b,c\in V$ (associativity).\\
		\item[](V2) $a \otimes (\bigvee S)=\bigvee_{s\in S}(a \otimes s)$ for every $S\subseteq V$ and every $a \in V$.\\
		\item[](V3) $(\bigvee S)\otimes a=\bigvee_{s\in S}(s\otimes a)$ for every $S\subseteq V$ and every $a \in V$.\\
		\item[](V4) $a \otimes e = e \otimes a=a$ for all $a\in V$ (unitality).\\
	\end{itemize}
	\end{definition}

A unital quantale $\mathbf V$ is called {\em commutative} if 
$u\otimes v=v\otimes u$ for all $u, v\in V$.

A quantale in general is a pair $\mathbf V=(V,\bigvee, \otimes)$, which doesn't have to satisfy unitality. Throughout this work, we will only consider unital quantales. For convenience, we may sometimes denote a unital quantale simply as $V$ instead of $\mathbf V$, but from the context, the meaning will always be clear.

\begin{example}
  \begin{itemize}[label=\textbullet, leftmargin=*]
   \item $\mujmsf{2}=(\{\bot_{\mujmsf{2}},\top_{\mujmsf{2}}\},\vee,\wedge, \top_{\mujmsf{2}})$ is a two-element unital quantale.
   \item $\mujmsf{P}_{+}=({[0,\infty]},\bigwedge, +,0)$ is a unital quantale given by the extended real half-line $[0,\infty]$ with the dual of the standard partial order.
  \end{itemize}
 \end{example}

 \begin{remark}
  \begin{itemize}[label=\textbullet, leftmargin=*]
   \item Every unital quantale $\mathbf V$ is a strict monoidal closed category.
   \item For every element $u\in V$, the map $V\arw{-\otimes u}V$ has a right adjoint map $V\arw{\homs(u,-)}V$ ($v\otimes u\leqs w$ iff
         $v\leqs\homs(u, w)$).
  \end{itemize}
 \end{remark}

We now define \emph{quantale modules}. A quantale module is a $\bigvee$-semilattice equipped with an action of a unital quantale, generalizing scalar multiplication in vector spaces. Notably, every $\bigvee$-semilattice is naturally a module over the two-element quantale $\mujmsf{2}$, allowing us to connect classical order-theoretic structures to our broader quantale framework.

	\begin{definition}\label{keymodule} Given a unital quantale $\mathbf V=(V,\bigvee, \otimes, e)$, a \textit{left ${\mathbf V}$-module} 
		is a triple $\mathbf A=(A,\bigvee,*)$ such that $(A,\bigvee)$ is a $\bigvee$-semilattice and $*\colon V\times A\longrightarrow A$ is a map satisfying:
	\begin{itemize}
		\item[](A1) $v * (\bigvee S)=\bigvee_{s\in S}(v * s)$ for every $S\subseteq A$ and every $v \in V$.\\
		\item[](A2) $(\bigvee T)* a=\bigvee_{t\in t}(t* a)$ for every $T\subseteq V$ and every $a \in A$.\\
		\item[](A3) $u*(v*a)=(u\otimes v)*a$ for every $u,v \in V$ and every $a\in A$.\\
		\item[](A4) $e * a=a$ for all $a\in A$ (unitality).\\
	\end{itemize}
		\end{definition}

A right \Vmodule{} can be defined in an analogical way and whenever we mention a 
\Vmodule{} in this work again, we will implicitly mean a left \Vmodule{} from the previous definition over a fixed but arbitrary unital quantale $\mathbf{V}=(V,\bigvee, \otimes, e)$.\\

\begin{example}
    Every quantale $\mathbf V=(V,\bigvee, \otimes,e)$ induces a \Vmodule{}, where the underlying set is the underlying set of $\mathbf V$ itself, namely $V$, the action $*$ is defined as $*:=\otimes$. In other words, the quantale $\textbf{V}=(V,\bigvee, \otimes,e)$ is itself a \Vmodule{} of the same form, as then the conditions for the \Vmodule{} and the quantale coincide.
\end{example}

\begin{definition}Let $\mathbf V=(V,\bigvee, \otimes,e) $ be a unital quantale. Given two left {\Vmodule{}s} $\mathbf A=( A,*_A)$, $\mathbf B=(B,*_B)$, a map $f:A\longrightarrow B$ is then called a \textit{left \Vmodule{} homomorphism} provided that it preserves all joins and $f(v*_Aa)=v*_Bf(a)$ for every $a\in A$ and every $v\in V$. 
	
	We denote by \VSUPH\ the category of (left) {\Vmodule{}s} and 
	(left) \Vmodule{} homomorphisms.
		\end{definition}
	
	To build a unifying categorical framework for various algebraic and logical concepts, we also introduce the related notions of \Vframe{}s (frames over a set with a \(\mathbf{V}\)-valued relation) 
    and \VFsemilattice{}s (\Vmodule{}s with an additional  endomorphism of modules). These definitions will be employed extensively throughout the paper.
	
		\begin{definition}\label{keyframe} Let $\mathbf V$ be a unital quantale. A \textit{\Vrelation{}} 
			$r$ from set $X$ to set $Y$ is a map 
   $r\colon X\times Y \longrightarrow V$ and  
			a \textit{\Vframe{}} over a set $T$ is a pair $(T,r)$ where $r$ is a map $r\colon T\times T\longrightarrow V$.
	
	\end{definition}
\begin{definition}Given two {\Vframe{}s} $(T,r)$ and $(S,s)$, 
a map $f\colon T\longrightarrow S$ is called 
	a \textit{\Vframe{} homomorphism} if it satisfies $r(i,j))\leq s(f(i),f(j))$ for every $i,j\in T$. We denote the category of \Vframe{}s as $\VJ$.
\end{definition}

	\begin{definition}\label{keyFsemi} Let $\mathbf V$ be a unital quantale. A \textit{\VFsemilattice{}} 
		is a pair $(\mathbf A,F)$ where $\mathbf A=(A,\bigvee, *)$ 
        is a {\Vmodule{}} and 
		$F\colon A \longrightarrow  A$ is a \Vmodule{} homomorphism.
\end{definition}

\begin{definition}	
	Given  two {\VFsemilattice{}s} $(\mathbf A_1,F)$ and $(\mathbf A_2,H)$, a \Vmodule{} 
	homomorphism $f:\mathbf A_1\longrightarrow \mathbf A_2$ is called 
	\begin{enumerate}[label=(\roman*)]
	\item a {\em \VFsemilattice{} homomorphism between $(\mathbf A_1,F)$ and $(\mathbf A_2,H)$} if it satisfies $H(f(a))=f(F(a))$ for any $a\in  A_1$, 
 \item a {\em lax morphism between  \VFsemilattice{}s $(\mathbf A_1,F)$ and $(\mathbf A_2,H)$} if it satisfies $H(f(a))\leq f(F(a))$ for any $a\in  A_1$.
	\end{enumerate}

We denote by \VFSUPH\ the category of {\VFsemilattice{}s} and 
\VFsemilattice{} homomorphisms and by \VFSUPL\ 
the category of {\VFsemilattice{}s} and lax morphisms.
 
\end{definition}

Recall that \VFSUPH\ is a subcategory of \VFSUPL{} and \VFSUPH\ is  
 an  equationally presentable category. Hence it has all small limits \cite[proposition 2.2.1]{KP}. These limits are 
constructed exactly as in the category of all sets. 

\begin{remark} If $\mathbf V=\mujmsf{2}$ then 
$\mujmsf{2}$-modules are exactly $\bigvee$-semilattices, 
    $\mujmsf{2}$-frames are exactly time frames and 
    $\mujmsf{2}-\textnormal{F}$-semilattices are exactly 
    $\textnormal{F}$-$\bigvee$-semilattices as studied in \cite{Bot}. 
    The seamless integration of classical and generalized contexts within the quantale approach serves to emphasize its natural character.
\end{remark}

\subsection{Quotients in \VSUPH{} and \VFSUPL{}}

In this subsection, we establish key technical foundations necessary for our later findings. We introduce and examine concepts including congruences on \Vmodule{}s and 
\VFsemilattice{}s, alongside the related notions of \VFprenuclei{} and \VFnuclei. These theoretical structures enable the formation of quotient structures and provide connections between congruences and closure operators within the context of sup-semilattices that carry a quantale action.

Specifically, for a given \Vmodule{} \(\mathbf{A}\) (where \(\mathbf{V}\) is a unital quantale), we define a congruence as an equivalence relation on \(A\) respecting both joins and the module action. We extend this to \VFsemilattice{}s, which feature an additional endomorphism \(F\). Furthermore, we introduce 
\VFprenuclei{} and demonstrate that their fixed point substructure (with induced operations) forms a \(\mathbf{V}\)-module and, indeed, a \VFsemilattice{}.
    

Let $\mathbf A$ be a \Vmodule{} where $\mathbf V$ is a unital quantale. 	
A { \em  congruence on the  \Vmodule{}}  $\mathbf A$ is 
	an equivalence relation $\theta$ on $A$ satisfying 
 \begin{enumerate}
	\item $\{(x_i, y_i)\mid i\in I\}\subseteq \theta$ implies  
	$(\bigvee_{i\in I}x_i)\theta (\bigvee_{i\in I}y_i)$, and
 \item for every $v \in V$ 
	and $(x, y)\in \theta$,  $(v*x, v*y)\in \theta$.
 \end{enumerate}
 Let us denote the set of all congruences on $\mathbf A$ as 
	$\mathrm{Con}\, \mathbf A$.

	If $\mathbf A$ is a \Vmodule{} and $\theta$ a \Vmodule{} congruence on $\mathbf  A$, the factor set
	$A/\theta$ is a \Vmodule{} again, and the projection 
	$\pi\colon A\to  A/\theta$ is, therefore, a \Vmodule{} morphism.
	Recall that if $\theta _i\in \mathrm{Con}\, \mathbf A$ for all $i\in I$, then also 
	$\bigcap _i \theta _i \in \mathrm{Con}\, \mathbf A$.
	
	A { \em  \VFcongruence{} on the \VFsemilattice}  $({\mathbf A},F)$ is 
	a congruence $\theta$ on $\mathbf A$ satisfying $a\theta b$ implies 
	$F(a)\theta F(b)$ for all $a, b\in A$. Note that if $F=\mbox{\rm id}_A$, then any congruence  on $\mathbf A$ is 
	a \VFcongruence{}  on  $({\mathbf A},\mbox{\rm id}_A)$. 

 A { \em  \VFprenucleus{} on the \VFsemilattice}  $({\mathbf A},F)$ is an 
 operator $j\colon A\to A$ such that, for every $v\in V$ and $a, b\in A$,
 \begin{enumerate}[label=({N\arabic*}), ref=\emph{\alph*}]
        \item $a\leq j(a)$,
	\item $a\leq b$ implies $j(a)\leq j(b)$,
        \item $v*j(a)\leq j(v*a)$,
        \item $F(j(a))\leq j(F(a))$.
 \end{enumerate}

We put 
$$
\begin{aligned}
    &A_j=\{a\in A \mid j(a)=a\},\ \leq_{{\mathbf A}_j}=\leq \cap (A_j\times A_j),\ 
    \mbox{$\bigvee_{{\mathbf A}_j}=j\circ \bigvee$},\\ 
    &*_{{\mathbf A}_j}=j\circ *\ 
    \text{and}\ F_{{\mathbf A}_j}=j\circ F.
\end{aligned}
$$
A { \em  \VFnucleus{} on the \VFsemilattice}  $({\mathbf A},F)$ is a 
 \VFprenucleus{} $j\colon A\to A$ such that, for every  $a\in A$,
 \begin{enumerate}
	\item[(N5)] $j(a)= j(j(a))$.
 \end{enumerate}

\begin{lemma}\label{lemnucleus}
	Let  $j$ be a nucleus on a \VFsemilattice{} $({\mathbf A},F)$.
	Then $\mathbf{A}_j=(A_j,\bigvee_{{\mathbf A}_j}, *_{{\mathbf A}_j})$ is  a \Vmodule{} and $(\mathbf{A}_j, F_{{\mathbf A}_j})$  is 
 a \VFsemilattice{}. Moreover, 	 
	the surjection $j\colon A\to {A_j}$ is a homomorphisms of \VFsemilattice{}s. 
\end{lemma}
\begin{proof} Clearly, $(A_j,\bigvee_{{\mathbf A}_j}, *_{{\mathbf A}_j})$ is  a \Vmodule{} and 
$j\colon A\to {A_j}$ is a homomorphisms of \Vmodule{}s (see, e.g., \cite{KP}). From \cite[Lemma 2.5]{Bot} we conclude that 
$(\mathbf{A}_j, F_{{\mathbf A}_j})$  is  a \VFsemilattice{} and 
$j\colon A\to {A_j}$ is a homomorphisms of \VFsemilattice{}s.
\end{proof}

As with sup-semilattices and \Vmodule{}s, there is a one-to-one correspondence 
between \VFnuclei{} and \VFcongruence{}s on \VFsemilattice{}s. 
This correspondence is defined by the following maps:
$$
\begin{array}{r c l}
	j&\mapsto&\theta_j; \ \text{here}\ %
	\theta_j=\{(a, b)\in A\times A\mid j(a)=j(b)\}\\[0.2cm]
	\theta &\mapsto& j_{\theta};\ \text{here}\ %
	j_{\theta}(x)=\bigvee\{y\in A\mid x\theta y\}.\\
\end{array}
$$
We omit the straightforward verification of this correspondence.

\begin{lemma}\label{lemprenucleus}
	Let  $({\mathbf A},F)$ be  \VFsemilattice{}   and  $j$ a \VFprenucleus{} on $({\mathbf A},F)$. Then 
	the poset $\left(A_j,\leq_{{\mathbf A}_j}\right)$ is a closure system on 
        $(A,\leq)$, and the associated closure operator $\mbox{n}(j)$ is given by 
	$$
	\mbox{\rm n}(j)(a)=\bigwedge\{x\in A_j\mid a\leq x\}.
	$$
	Moreover, $\mbox{\rm n}(j)$ is  a \VFnucleus{} on $({\mathbf A},F)$. 
\end{lemma}
\begin{proof} From \cite[Lemma 2.6]{Bot} we  obtain that 
	$\left(A_j,\leq_{{\mathbf A}_j}\right)$ is a closure system and  $\mbox{n}(j)$ is a closure operator satisfying (N4).  From \cite[Proposition 2.3.4.]{KP} we have that $\mbox{n}(j)$ satisfies (N1), (N2), (N3) and (N5). We conclude that 
 $\mbox{\rm n}(j)$ is  a \VFnucleus{} on $({\mathbf A},F)$.
\end{proof}

We will need the following. 

Let  $({\mathbf A},F)$ be  a \VFsemilattice{}   and  $X\subseteq A^{2}$. We put 
$$
\begin{array}{r c l}
	j[X](a)&=&a\vee\bigvee\{c\in A\mid d \leq a, (c,d)\in X\ \text{or}\ (d,c)\in X\}.
\end{array}
$$

\begin{lemma}\label{lemprenucleus2}
	Let  $({\mathbf A},F)$ be  \VFsemilattice{}   and  
 $X\subseteq A^{2}$ a subset such that 
	$(F\times F)(X)\cup \{(v*c, v*d)\mid (c,d)\in X, v\in V\}\subseteq X$.  Then the mapping 
	$j[X]\colon A \to A$ is a prenucleus on $({\mathbf A},F)$.  Moreover, 
	for every \VFsemilattice{}   $({\mathbf B},F_{{\mathbf B}})$ and every lax morphism 
	$g\colon A\to B$ of \VFsemilattice{}s such that $g(c)=g(d)$ for all $(c,d)\in X$  there is a unique 
	lax morphism $\overline{g}\colon A_{\mbox{\rm\small n}(j[X])}\to B$ 
	of \VFsemilattice{}s such that $g=\overline{g}\circ \mbox{\rm n}(j[X])$. 
\end{lemma}
\begin{proof} From \cite[Lemma 2.7]{Bot} we  obtain that $j[X]$ satisfies 
(N1), (N2)  and (N4). It remains to check that it satisfies (N3). 
	
 Let $a\in A$ and $v\in V$. We have to show that $v*j[X](a)\leq j[X](v*a)$. 
 Evidently, $v*a\leq j[X](v*a)$. 
 Let $c\in A$ such that $(c, d)\in X$ and $d\leq a$. Then 
 $(v*c, v*d)\in X$  and $v*d\leq v*a$. We conclude that 
 $v*c\leq j[X](v*a)$. Similarly, if $c\in A$ such that $(d, c)\in X$ and $d\leq a$ we 
 have $v*c\leq j[X](v*a)$. Therefore, 
 $v*j[X](a)=v*a\vee\bigvee\{v*c\in A\mid d \leq a, (c,d)\in X\ \text{or}\ (d,c)\in X\}\leq j[X](v*a)$. 

 By applying \cite[Lemma 2.7]{Bot}, we obtain a map 
 $\overline{g}\colon A_{\mbox{\rm\small n}(j[X])}\to B$. This map $\overline{g}$ has the following properties:
\begin{enumerate}
\item It satisfies the equation $g = \overline{g} \circ \mbox{\rm n}(j[X])$.
\item It preserves arbitrary joins.
\item For all elements $a$ in $A_{\mbox{\rm\small n}(j[X])}$, the inequality $F_H(\overline{g}(a)) \leq \overline{g}(F_{\mbox{\rm\small n}(j[X])}(a))$ holds.
\end{enumerate}

It remains to show that $\overline{g}$ is a \Vmodule{} homomorphism. 
Let $a \in A_{\mbox{\rm\small n}(j[X])}$ and $v\in V$. We compute:
\begin{align*}
    \overline{g}(v*_{{\mathbf A}_{\mbox{\rm\tiny n}(j[X])}} a)&=%
    \overline{g}(\mbox{\rm n}(j[X])(v*_A a))=g(v*_A a)=v*_B g(a)\\
    &=v*_B \overline{g}(\mbox{\rm n}(j[X])(a))= v*_B \overline{g}(a).
\end{align*}	
\end{proof}
	
\begin{remark} Recall that all the above statements can be easily extended to $\bigvee$-algebras 
\cite{Resende, Paseka}. $\bigvee$-algebras provide a way to study $\bigvee$-semilattices with additional algebraic structure, bridging order theory and universal algebra.
\end{remark}

	\section{Basic constructions and functorial theorems}\label{Basic constructions and functorial theorems}

 This section lays out the fundamental concepts and constructs that will serve as the cornerstones of our future work. 

 From now on we will always assume that $\mathbf V$ is a commutative unital quantale.

 We present now the foundational information for three key constructions:
\begin{enumerate}
\item $( - )^{( - )}\colon   \VSUPH \times \VJ \to \VFSUPL$,
\item $( - ) \otimes ( - )\colon  \VJ \times  \VFSUPL \to \VSUPH$,
\item ${\mathbf J}[( - ), ( - )]\colon \VFSUPL \times \VSUPH \to \VJ$.
\end{enumerate}

These functors are central to our work, giving rise to three adjoint situations $(\eta, \varepsilon), (\varphi, \psi)$, and $(\nu, \mu)$ (see Section~\ref{Three Adjoint Situations}). The functorial theorems presented here hold for any general commutative unital quantale $\mathbf{V}$ and the specialized versions for $\mathbf{V}=\mujmsf{2}$ align with theorems established in the paper \cite{Bot}.

 First, we prescribe to a \Vmodule\ $\mathbf{A}$ and a 
 {\Vframe} $\mathbf{J}$ a {\VFsemilattice} $\mathbf A^{\mathbf J}$.

		\begin{definition}\label{keyFsemiprod}
	Let $\mathbf A=(A,\bigvee, *)$ be a \Vmodule{} and $\mathbf J=(T,r)$ be a {\Vframe{}}. Let us define a \textit{\VFsemilattice{}} $\mathbf A^{\mathbf J}$ as $(\mathbf A^T,F^{\mathbf J})$ where
	
	$$(F^{\mathbf J}(x))(i)=\bigvee \{r(i,k)*x(k) \mid {k\in T}\}.$$
	
	The \Vmodule{} action on the \Vmodule{} ${\mathbf A}^T$, denoted by $*^T$, is defined, for any pair $(v,x)$, as  $(v*^T x)(t)=v*x(t)$ for every $v \in V$, $x \in  A^T$ and $t\in T$ .
\end{definition}

 To ensure the correctness of Definition \ref{keyFsemiprod}, we know that 
 \VSUPH{}  has all small limits, and they are constructed exactly as in the category of sets (see \cite[Proposition 2.2.1 and Theorem 2.2.2]{KP}). Hence  
 ${\mathbf A}^T$ is a \Vmodule{}. We have to verify that 
 $F^{\mathbf J}$ is a \Vmodule{} homomorphism. 
 
 Let $U\subseteq {A}^T$, $z\in {A}^T$, $v\in V$ and $i\in T$. We compute:
 \begin{align*}
     (F^{\mathbf J}(\bigvee U))(i)&=%
     \bigvee \{r(i,k)*(\bigvee U)(k) \mid {k\in T}\}\\
     &=%
     \bigvee \{r(i,k)*(\bigvee \{x\mid x\in U\})(k) \mid {k\in T}\}\\
     &=\bigvee \{r(i,k)*x(k) \mid {k\in T}, x\in U\}\\
     &=\bigvee \{\bigvee \{r(i,k)*x(k) \mid {k\in T}\}\mid x\in U\}\\
     &=\bigvee \{(F^{\mathbf J}(x))(i)\mid x\in U\}%
     =\left(\bigvee \{F^{\mathbf J}(x)\mid x\in U\}\right)(i)
 \end{align*}
 and 
 \begin{align*}
(F^{\mathbf J}&(v*^T z))(i)=%
\bigvee \{r(i,k)*(v*^T z)(k) \mid {k\in T}\}\\
     &=
     \bigvee \{r(i,k)*(v* z(k)) \mid {k\in T}\}
     =%
     \bigvee \{(r(i,k)\otimes v)* z(k)\mid {k\in T}\}\\
     &=%
     \bigvee \{(v\otimes r(i,k))* z(k) \mid {k\in T}\}
     =%
      \bigvee \{v* (r(i,k)* z(k)) \mid {k\in T}\}\\
     &=%
      v* (\bigvee \{r(i,k)* z(k)\mid {k\in T}\})
     =%
     v* (F^{\mathbf J}(x))(i)=(v*^T (F^{\mathbf J}(x)))(i).
 \end{align*}
 
	
		It's worth noting that Definition \ref{keyFsemiprod} represents a broader concept compared 
        to a definition we employed in our previous work (\cite[Deﬁnition 3.2.]{Bot}):

	Let $\mathbf L=(L,\bigvee)$ be a $\bigvee$-semilattice 
	and ${\mathbf J}=(T,S)$ a frame. Let us define 
	an $F$-$\bigvee$-semilattice ${\mathbf L}^{\mathbf J}$ as ${\mathbf L}^{\mathbf J}=(\mathbf L^T,F^{\mathbf J})$, 
	where $$(F^{\mathbf J}(x))(i)=\bigvee \{x(k) \mid {(i,k)\in S}\}$$ 
	for all $x\in L^T$. $F^{\mathbf J}$ will be called an 
	{\em operator on} ${\mathbf L}^{T}$ 
	{\em constructed by means of the frame} ${\mathbf J}$.

        Since every \Vframe{} boils down to a time frame 
        and a \Vmodule{} becomes an usual $\bigvee$-semilattice{} when 
        $\mathbf V$ is 
        a trivial two-element quantale $\mujmsf{2}$ both definitions coincide. 
        
	
	\begin{proposition}\label{funSSF}
			Let $\mathbf A_1$ and $\mathbf  A_2$ be \Vmodule{}s, 
		let $f\colon \mathbf A_1\longrightarrow\ \mathbf A_2$ be a homomorphism 
		of \Vmodule{}s, and let $\mathbf J=(T,r)$ be a \Vframe{}. Then 
		there exists a homomorphism 
		$f^{\mathbf J}\colon\mathbf A_1^{\mathbf J}\longrightarrow\ \mathbf A_2^{\mathbf J}$ 
		of\/ {\VFsemilattice{}s} such that, for every $x\in A_1^T$ and every $i\in T$, it holds $(f^{\mathbf J}(x))(i)=f(x(i))$. Moreover, $-^{\mathbf J}$ is a functor from \VSUPH{} to \VFSUPH.
	\end{proposition}

	\begin{proof} Since  \VSUPH\ has arbitrary products, we know that 
		$f^{\mathbf J}\colon \mathbf A_1^{T}\longrightarrow\mathbf A_2^{T}$ 
		is a morphism in \VSUPH. We denote 
  $\mathbf A_1^{\mathbf J}=(\mathbf A_1^{T}, F^{\mathbf J}_{1})$ and 
  $\mathbf A_2^{\mathbf J}=(\mathbf A_2^{T}, F^{\mathbf J}_{2})$.
		It remains to check that 
		$f^{\mathbf J}\colon\mathbf A_1^{\mathbf J}\longrightarrow\mathbf A_2^{\mathbf J}$ 
		is a homomorphism of {\VFsemilattice{}s}. We compute: 
		$$\begin{array}{@{}r@{\,\,}c@{\,\,}l}
			(F^{\mathbf J}_2(f^{\mathbf J}(x)))(i)&%
   =&\bigvee_{k\in T} (r(i,k)*_{\mathbf A_2}f^{\mathbf J}(x)(k))%
   = \bigvee_{k\in T} r(i,k)*_{\mathbf A_2}f(x(k))\\[0.2cm]
			&=&f(\bigvee_{k\in T} r(i,k)*_{\mathbf A_1}x(k))%
   =f(F^{\mathbf J}_1(x)(i))=(f^{\mathbf J}(F^{\mathbf J}_1(x)))(i).\end{array}$$
   The functoriality of $-^{\mathbf J}$ follows from the functoriality of the product construction in $V$-$\mathbb S$.
		\end{proof}
		
			\begin{proposition}\label{proptJ}
		Let $t\colon\mathbf  J_1\longrightarrow\ \mathbf J_2$ be a homomorphism of \Vframe{}s 
		$\mathbf J_1=(T_1,r)$ and $\mathbf J_2=(T_2,s)$, 		 
		and $\mathbf A$ a \Vmodule{}. Then there exists a lax morphism 
		${\mathbf A}^t\colon {\mathbf A}^{{\mathbf J}_2}\longrightarrow\ {\mathbf A}^{\mathbf J_1}$  
		of\/ {\VFsemilattice{}s} such that, for every $x\in A^{T_2}$ and every $i\in T_1$, it holds 
		$({\mathbf A}^t(x))(i)=x(t(i))$. Moreover, ${\mathbf A}^{-}$ is a contravariant functor from $\VJ$ 
		to the category \VFSUPL{} of\/ \VFsemilattice{}s.
			\end{proposition}
		
		\begin{proof} Evidently, 
				${\mathbf A}^t\colon {\mathbf A}^{{T}_2}\longrightarrow\ {\mathbf A}^{T_1}$ is 
				a homomorphism of \Vmodule{}s. Let us show that it is lax. 
				Assume that $x\in A^{T_2}$ and $i\in T_1$. Then 
			$$\begin{array}{r@{\,\,}c@{\,\,}l}
		(F^{{\mathbf J}_1}({\mathbf A}^{t}(x)))(i)&=&
		\bigvee_ {k\in T_1} r(i,k) *_{{\mathbf A}}{\mathbf A}^{t}(x)(k)%
  =\bigvee_ {k\in  T_1} r(i,k)*_{{\mathbf A}} x(t(k))\\[0.2cm]
		&\leq&\bigvee_ {k\in  T_1} s(t(i),t(k))*_{{\mathbf A}} x(t(k))%
  \leq \bigvee_ {l\in  T_2} s(t(i),l)*_{{\mathbf A}} x(l)\\[0.2cm]
		&=&(F^{\mathbf J_2}(x))(t(i))%
		={\mathbf A}^t (F^{\mathbf J_2} (x))(i)\\
	\end{array}$$
 for all $x\in A^{T_2}$ and  $i\in T_1$. Now, let 
	$s\colon\mathbf  J_0\longrightarrow\ \mathbf J_1$ 
 be a homomorphism of \Vframe{}s 
	and let   $x\in A^{T_2}$ and $i\in T_0$. We compute:
	$$
	{\mathbf A}^{t\circ s}(x)(i)=x(t(s(i))={\mathbf A}^{t}(x)(s(i))={\mathbf A}^{s}({\mathbf A}^{t}(x))(i)=%
	\left(({\mathbf A}^{s}\circ {\mathbf A}^{t})(x)\right)(i).%
	$$
	Clearly, ${\mathbf A}^{{\mathrm id}_{\mathbf J_2}}(x)(i)=x(i)={\mathrm id}_{{\mathbf A}^{\mathbf J_2}}(x)(i)$ 
	for all  $x\in A^{T_2}$ and  $i\in T_2$. Hence ${\mathbf A}^{-}$ is really a contravariant functor from $V$-$\mathbb{J}$ 
	to $V$-$\FSUPL$.
		\end{proof}

In what follows we will construct a \Vmodule\  $\mathbf{J} \otimes \mathbf{H}$ from a {\Vframe} $\mathbf{J}$ and 
    a {\VFsemilattice} $\mathbf{H}$.

	Let $\mathbf A$ be a \Vmodule{} and ${\mathbf J}=(T,r)$ 
	a \Vframe{}. Then, 
	for arbitrary $x\in A$ and $(i,j)\in T\times T$, we define $x_{ir}(j)=r(i,j)*x$ and $x_{i=}$ by
	$$
	x_{i=}(j)=\left\{
	\begin{array}{l l}
		x&\text{if}\ i=j; \\
		0&\text{otherwise.}
	\end{array}\right.
	$$

 Note that 
 \begin{align*}
     x_{ir}=\bigvee\{r(i,k)*x_{k=}\mid k\in T\}.
 \end{align*}

	Let $\mathbf H=(\mathbf A, F)$ be an \VFsemilattice{}  
	and ${\mathbf J}=(T,r)$ a \Vframe{}. We put 
	$$
	[\mathbf J, \mathbf H]=\{(x_{ir}\vee F(x)_{i=}, F(x)_{i=})\mid x\in A, i\in T\}.
	$$
	
	Then $\{(v*^{T} (x_{ir}\vee F(x)_{i=}), v*^{T} F(x)_{i=}) \mid x\in A, i\in T\} \subseteq [\mathbf J, \mathbf H]\subseteq A^{T}\times  A^{T}$. 
 By Lemma \ref{lemprenucleus2} applied to \Vmodule{}s we can define a \Vmodule{} 
	${\mathbf J}\otimes {\mathbf H}$ as follows: 
	$${\mathbf J}\otimes {\mathbf H}={\mathbf A}^T_{j[\mathbf J, \mathbf H]}.$$

Let $\mathbf J_1=(T_1,r)$ and $\mathbf J_2=(T_2,s)$ be \Vframe{}s, 
$f:T_1\longrightarrow T_2$  a \Vframe{} homomorphism, and 
$(A,\bigvee, *)$ a \Vmodule{}.

We define a {\em forward operator} 
$f^{\rightarrow}\colon A^{T_1}\longrightarrow A^{T_2}$ 
evaluated on $k\in T_2$ for any $x\in A^{T_1}$ as follows:
$$(f^{\rightarrow}(x))(k)=\bigvee \{x(i)\mid f(i)=k \}$$
where $k\in T_2$. From \cite[Theorem 2.9]{Rod} we know that 
$P\colon \VJ\to \dvaSUPH$ defined by $P(\mathbf J_1)={\mathbf A}^{T_1}$ 
and $P(f)=f^{\rightarrow}$ is a functor.

   \begin{theorem}
Let $\mathbf J_1=(T_1,r)$ and $\mathbf J_2=(T_2,s)$ be \Vframe{}s, 
$f\colon T_1\longrightarrow T_2$  a \Vframe{} homomorphism, and 
$\mathbf A=(A,\bigvee, *)$ a \Vmodule{}. The operator \( -^{\rightarrow} \) defined as \( -^{\rightarrow} (\mathbf J_1) = \mathbf A^{T_1} \) on objects and \( -^{\rightarrow} (f) = f^{\rightarrow} \) on morphisms is a functor from the category \VJ{} of\/ 
\Vframe{}s to the category of \VSUPH.
   \end{theorem}

   \begin{proof} First, we show that $f^{\rightarrow}$ is a \Vmodule{} 
   homomorphism. Let  $x\in A^{T_1}$, $v\in V$ and $k\in T_2$. We compute: 
   \begin{align*}
       (v*^{T_2} f^{\rightarrow}(x))(k)&=v*(f^{\rightarrow}(x)(k))%
       =v*\bigvee \{x(i)\mid f(i)=k \}\\
       &=%
       \bigvee \{v*x(i)\mid f(i)=k \}= \bigvee \{(v*^{T_1}x)(i)\mid f(i)=k \}\\
       &=f^{\rightarrow}(v*^{T_1}x)(k).
   \end{align*}

To prove that the operator \( -^{\rightarrow} \) is a functor it is enough to note 
that $P$ is a functor and that range of $P$ is contained in \VSUPH.
\end{proof}

	\begin{proposition}\label{proptJfac}
 Let 	$f\colon{}\mathbf J_1\longrightarrow \mathbf J_2$ be a homomorphism of \Vframe{}s $\mathbf J_1=(T_1,r_1)$ and $\mathbf J_2=(T_2,r_2)$, 
 and $\mathbf H=(\mathbf A, F)$ 
	a \VFsemilattice. Then there exists a unique morphism 
	$f\otimes \mathbf H\colon \mathbf J_1\otimes \mathbf H\to\mathbf  J_2\otimes \mathbf H$ 
	of \Vmodule{}s such that the following diagram commutes:
	\begin{center}
		
		\begin{tikzpicture}
			\node (alfa') at (7,5) {$\mathbf A^{T_1}$};
			\node (beta') at (7,2) {$\mathbf A^{T_2}$};
			\node (gama') at (12,2) {$\mathbf J_2\otimes \mathbf H$};
			\node (delta') at (12,5) {$\mathbf J_1\otimes \mathbf H$};
			
			\node (1') at (9.5,4.5) {$\mbox{\rm n}(j[\mathbf J_1,\mathbf H])$};
			\node (2') at (9.5,2.5) {$\mbox{\rm n}(j[\mathbf J_2,\mathbf H])$};
			\node (3') at (7.4,3.55) {$f^{\rightarrow}$};
			\node (4') at (11.2,3.55) {$f\otimes \mathbf H$};
			
			\draw [->](alfa') -- (delta');
			\draw [->](beta') -- (gama');
			\draw [->](delta') -- (gama');
			\draw [->](alfa') -- (beta');
		\end{tikzpicture}
\end{center}
	Moreover, then $-\otimes\mathbf  H$ is a functor from $V$-$\mathbb{J}$ to $V$-${\mathbb S}$.

\end{proposition}
\begin{proof} Let $x\in A$ and $i\in T_1$. We have to verify
	\begin{align*}
		\mbox{\rm n}(j[\mathbf J_2,\mathbf H])(f^{\rightarrow}(F(x)_{i=}))=%
		\mbox{\rm n}(j[\mathbf J_2,\mathbf H])(f^{\rightarrow}(x_{ir_1}\vee F(x)_{i=})).
	\end{align*}
	
	We compute:
	\begin{align*}
		(f^{\rightarrow}(x_{ir_1}))(l)&=\bigvee \{(x_{ir_1}(k)\mid f(k)=l \}=%
		\bigvee \{(r_1(i,k)*x\mid f(k)=l \}\\ 
&\leq \bigvee \{r_2(f(i),f(k))*x\mid f(k)=l \}=r_2(f(i),l)*x= x_{f(i)r_2}(l)
	\end{align*}

for all $l\in T_2$. Hence $f^{\rightarrow}(x_{ir_1})\leq x_{f(i)r_2}$. 
Similarly, we have 
\begin{align*}
f^{\rightarrow}(F(x)_{i=})(l)=\bigvee  \{F(x)_{i=}(k) \mid f(k)=l\}=F(x)_{f(i)=}(l),
	\end{align*}
i.e.,  $f^{\rightarrow}(F(x)_{i=})=F(x)_{f(i)=}$.

Therefore it holds 
\begin{align*}
	\mbox{\rm n}(j[\mathbf J_2,\mathbf H])&(f^{\rightarrow}(x_{i{r_1}}))\leq %
	\mbox{\rm n}(j[\mathbf J_2,\mathbf H])(x_{f(i)r_2})\\
        &\leq %
	\mbox{\rm n}(j[\mathbf J_2,\mathbf H])(x_{f(i)r_2}\vee F(x)_{f(i)=})= %
	\mbox{\rm n}(j[\mathbf J_2,\mathbf H])(F(x)_{f(i)=})\\
 &=%
	\mbox{\rm n}(j[\mathbf J_2,\mathbf H])(f^{\rightarrow}(F(x)_{i=})).%
\end{align*}

We conclude 
$$\begin{aligned}
	\mbox{\rm n}(j[\mathbf J_2,\mathbf H])&%
	(f^{\rightarrow}(x_{ir_1}\vee F(x)_{i=}))\\%
	&\leq\mbox{\rm n}(j[\mathbf J_2,\mathbf H])(f^{\rightarrow}(x_{ir_1}))\vee %
	\mbox{\rm n}(j[\mathbf J_2,\mathbf H])(f^{\rightarrow}(F(x)_{i=}))\\%
    &\leq\mbox{\rm n}(j[\mathbf J_2,\mathbf H])(x_{f(i)r_2}))\vee %
	\mbox{\rm n}(j[\mathbf J_2,\mathbf H])(F(x)_{f(i)=})\\%
	&=\mbox{\rm n}(j[\mathbf J_2,\mathbf H])(F(x)_{f(i)=})= \mbox{\rm n}(j[\mathbf J_2,\mathbf H])(f^{\rightarrow}(F(x)_{i=}))\\%
	&\leq \mbox{\rm n}(j[\mathbf J_2,\mathbf H])(f^{\rightarrow}(x_{ir_1}\vee F(x)_{i=})).
\end{aligned}
$$

If we apply Lemma \ref{lemprenucleus2} we get a unique morphism $f\otimes \mathbf H$ 
of \Vmodule{}s such that 
$$\mbox{\rm n}(j[\mathbf J_2,\mathbf H])\circ f^{\rightarrow}=(f\otimes \mathbf H)\circ %
\mbox{\rm n}(j[\mathbf J_1,\mathbf H]).$$

Let us show that $-\otimes\mathbf  H$ is a functor. Evidently, for  
	$\mathbf J=(T,r)$ and $\mbox{\rm id}_T$ we have that 
	$\mbox{\rm id}_T^{\rightarrow}=\mbox{\rm id}_{\mathbf A^{T}}$.
	Hence $\mbox{\rm id}_T\otimes  \mathbf H=\mbox{\rm id}_{\mathbf  J\otimes\mathbf  H}$. 
	Now, let $t\colon{}\mathbf J_1\to\ \mathbf J_2$  and $s\colon{}\mathbf J_2\to\ \mathbf J_3$ 
	be  homomorphisms of frames. Then $(s\circ t)^{\rightarrow}=s^{\rightarrow}\circ t^{\rightarrow}$. From the uniqueness 
	property of $t\otimes\mathbf  H$, $s\otimes\mathbf  H$ and $(s\circ t)\otimes\mathbf  H$, we obtain that $(s\circ t)\otimes\mathbf  H=(s\otimes\mathbf  H)\circ (t\otimes\mathbf  H)$.
\end{proof}

\

\begin{proposition}\label{nucn}Let $\mathbf H_1=(\mathbf A_1, F_1), \mathbf H_2=(\mathbf A_2, F_2)$ be 
	\VFsemilattice{}s,  $f\colon \mathbf H_1 \to \mathbf H_2$ 
 a  lax morphism of\/ \VFsemilattice{}s
	and $\mathbf J=(T,r)$ a \Vframe{}. Then there is a unique morphism 
	$\mathbf J\otimes f\colon \mathbf J\otimes \mathbf H_1\to\ \mathbf J\otimes \mathbf H_2$ of 
	\Vmodule{}s such that the following diagram commutes:
	
	\begin{center}
		\begin{tikzpicture}
			
			\node (alfa') at (7,5) {$\mathbf A_1^{T}$};
			\node (beta') at (7,2) {$\mathbf A_2^{T}$};
			\node (gama') at (12,2) {$\mathbf J\otimes \mathbf H_2$};
			\node (delta') at (12,5) {$\mathbf J\otimes \mathbf H_1$};
			
			\node (1') at (9.5,4.5) {$\mbox{\rm n}(j[\mathbf J,\mathbf H_1])$};
			\node (2') at (9.5,2.5) {$\mbox{\rm n}(j[\mathbf J,\mathbf H_2])$};
			\node (3') at (7.45,3.55) {$f^{\mathbf J}$};
			\node (4') at (11.2,3.55) {$\mathbf J\otimes f$};
			
			\draw [->](alfa') -- (delta');
			\draw [->](beta') -- (gama');
			\draw [->](delta') -- (gama');
			\draw [->](alfa') -- (beta');
		\end{tikzpicture}
	\end{center}
	
	Moreover, $\mathbf J\otimes -$ is a functor from the category \VFSUPL\ of \/
 \VFsemilattice{}s to the category  \VSUPH\ of \Vmodule{}s.
\end{proposition}

\begin{proof} Let $x\in A_1$ and $i, k\in S$. 
By definition, we have $(f^{\mathbf J}(x))(i)=f(x(i))$, so 
$$(f^{\mathbf J}(x_{ir}))(k)=f(x_{ir}(k))=f(r(i,k)*x)=r(i,k)*f(x)=f(x)_{ir}(k).$$

We compute:
		\begin{align*}
	\mbox{\rm n}&(j[\mathbf J,\mathbf H_2])(f^{\mathbf J}(x_{i{r}}))= %
	\mbox{\rm n}(j[\mathbf J,\mathbf H_2])(f(x)_{ir})\\
 &\leq %
 \mbox{\rm n}(j[\mathbf J,\mathbf H_2])(f(x)_{ir} \vee (F_2(f(x)))_{i=})=
	\mbox{\rm n}(j[\mathbf J,\mathbf H_2])(F_2(f(x))_{i=})\\
	&\leq %
	\mbox{\rm n}(j[\mathbf J,\mathbf H_2])(f(F_1(x))_{i=})=%
	\mbox{\rm n}(j[\mathbf J_2,\mathbf H])(f^{\mathbf J}(F_1(x)_{i=})).%
\end{align*}
	
	Hence 
	$$
	\mbox{\rm n}(j[\mathbf J,\mathbf H_2])(f^{\mathbf J}(x_{ir}\vee F_1(x)_{i=}))=%
	\mbox{\rm n}(j[\mathbf J,\mathbf H_2])(f^{\mathbf J}(F_1(x)_{i=})).
	$$
	We conclude, as before  from Lemma \ref{lemprenucleus2},  that there is a unique morphism 
	$\mathbf J\otimes f\colon \mathbf J\otimes \mathbf H_1\to\ \mathbf J\otimes \mathbf H_2$
	of \Vmodule{}s such that 
	$\mbox{\rm n}(j[\mathbf J,\mathbf H_2])\circ f^{\mathbf J}=(\mathbf J\otimes f)\circ %
	\mbox{\rm n}(j[\mathbf J,\mathbf H_1])$

 Let us verify that $\mathbf J\otimes -$ is a functor. 
	Let $\mathbf H=(\mathbf A, F)$ be a \VFsemilattice{} and 
	$\mbox{\rm id}_{\mathbf H}$ the identity morphism on $\mathbf H$.

	We know from Theorem  \ref{funSSF} that $-^{\mathbf J}$ is a functor from 
 \VSUPH{} to \VFSUPH{}, 
	hence also from  \VFSUPL{} to \VFSUPH{} (we 
 take twice the forgetful functor  from  \VFSUPL{} to \VSUPH{}).
	Hence  we have that 
	$\mbox{\rm id}_{\mathbf H}^{\mathbf J}=\mbox{\rm id}_{\mathbf A^{T}}$. Therefore 
	$\mathbf J\otimes \mbox{\rm id}_{\mathbf H}=\mbox{\rm id}_{\mathbf J\otimes \mathbf H}$.

	Now, let $f\colon{}\mathbf H_1\to\ \mathbf H_2$  and $g\colon{}\mathbf H_2\to\ \mathbf H_3$ 
	be   lax morphisms of \VFsemilattice{}s. 
	Then $(g\circ f)^{\mathbf J}=g^{\mathbf J}\circ f^{\mathbf J}$. From the uniqueness 
	property of $\mathbf J\otimes  f$, $\mathbf J\otimes g$ and $\mathbf J\otimes (g\circ f)$ we conclude  that 
	$\mathbf J\otimes (g\circ f)=(\mathbf J\otimes g)\circ (\mathbf J\otimes  f)$.
\end{proof}

Let $\mathbf L$ be a \Vmodule{} and $\mathbf{H}=(\mathbf A,F)$ a \VFsemilattice. 

The last construction we will need gives us 
a {\Vframe} $\mathbf{J}[\mathbf{H}, \mathbf{L}]$ from a {\VFsemilattice} $\mathbf{H}$ 
    and a \Vmodule\   $\mathbf{L}$.

Let us define $a \rightarrow b$ for $a,b\in L$ as $\bigvee\{v\in V \mid v*a \leq b\}$. 

Let $T_{[\mathbf H, \mathbf L]}={\VSUPH}(\mathbf A, \mathbf L)$, i.e., 
	the elements of $T_{[\mathbf H, \mathbf L]}$ are morphisms $\alpha$ of \Vmodule{}s
	from $\mathbf A$ to $\mathbf L$. Let $\alpha, \beta \in T_{[\mathbf H, \mathbf L]}$ and $v\in V$. Assume, for all $x \in A$, that $v \otimes \beta(x) \leq \alpha(F(x))$.  
This means $v \leq \beta(x) \rightarrow \alpha(F(x))$ for all $x \in A$.

We then put  $r(\alpha, \beta)=\bigwedge_{x \in A} \{\beta(x) \rightarrow \alpha(F(x))\}$.

\begin{definition}\label{framefromfss}  Let $\mathbf H=(\mathbf A, F)$ be a 
\VFsemilattice{} and let $\mathbf L$ be a \Vmodule{}. 
We define a \Vframe{} $\mathbf J[\mathbf H, \mathbf L]=	(T_{[\mathbf H, \mathbf L]},r)$. 
\end{definition}

\begin{theorem}\label{funSJrel} Let $\mathbf L_1, \mathbf L_2$ be \Vmodule{}s, 
	$\mathbf H=(\mathbf A, F)$ a \VFsemilattice,  
	and  let $f\colon \mathbf  L_1\to \ \mathbf L_2$ be a morphism of\/ \Vmodule{}s. Then there exists 
	a homomorphism $\mathbf J[\mathbf H, f]\colon \mathbf J[\mathbf H, \mathbf L_1]\to \mathbf J[\mathbf H, \mathbf L_2]$ 
	of\/ \Vframe{}s such that 
	$$(\mathbf J[\mathbf H, f](\alpha))(x)=f(\alpha(x))$$ 
	for all $\alpha\in T_{[\mathbf H, \mathbf L_1]}$ and all $x\in A$. Moreover, $\mathbf J[\mathbf H, -]$ 
	is a functor from a category $\VSUPH$ of \Vmodule{}s to the category $\VJ$ of \Vframe{}s.
\end{theorem}

\begin{proof} We first show that, for every $\alpha\in T_{[\mathbf H, \mathbf L_1]}$, 
$(\mathbf J[\mathbf H, f](\alpha)\in T_{[\mathbf H, \mathbf L_2]}$.

Let $U\subseteq A$, $x\in A$ and $v\in V$. We compute: 
\begin{align*}
    (\mathbf J[\mathbf H, f](\alpha))(\bigvee U)&=f(\alpha(\bigvee U))=%
    \bigvee\{f(\alpha(u))\mid u\in U\}\\
    &= \bigvee\{  (\mathbf J[\mathbf H, f](\alpha))(u)\mid u\in U\}
\end{align*}
and 
\begin{align*}
   (\mathbf J[\mathbf H, f](\alpha))(v*x)&=f(\alpha(v*x))= %
   v* f(\alpha(x))=v* (\mathbf J[\mathbf H, f](\alpha))(x).
\end{align*}

Now, we will verify that $\mathbf J[\mathbf H, f]$ is a homomorphism of 
\Vframe{}s. 
Let $\alpha, \beta\in T_{[\mathbf H, \mathbf L_1]}$. Then 	
$r_1(\alpha, \beta)*^{A}\beta \leq \alpha\circ F$ in $\mathbf L_1^{A}$. 



	For all $x \in A$, we compute:
\begin{align*}
         r(\alpha, \beta)*f(\beta(x)) =   f(r(\alpha,\beta)*\beta(x)) \leq f(\alpha(F(x)).
        \end{align*}
We conclude   $ r_1(\alpha, \beta)*\mathbf J[\mathbf H, f](\beta)(x) \leq \mathbf J[\mathbf H, f](\alpha)(F(x))$. Therefore, $ r_1(\alpha, \beta) \leq %
\mathbf J[\mathbf H, f](\beta)(x) \rightarrow \mathbf J[\mathbf H, f](\alpha)(F(x))$, 
i.e., $ r_1(\alpha, \beta) \leq %
r_2(\mathbf J[\mathbf H, f](\alpha), \mathbf J[\mathbf H, f](\beta))$.

	Let us check that  $\mathbf J[\mathbf H, -]$ is a functor. Let $\mathbf L$ be a \Vmodule{}, 
	$\mbox{\rm id}_{\mathbf L}$ the identity morphism of \Vmodule{}s on $\mathbf L$ and
	$\alpha\in T_{[\mathbf H, \mathbf L]}$. We compute:
	$$
	\mathbf J[\mathbf H, \mbox{\rm id}_{\mathbf L}](\alpha)=%
	\mbox{\rm id}_{\mathbf L} \circ \alpha=\alpha=\mbox{\rm id}_{\mathbf J[\mathbf H, \mathbf L]}(\alpha). 
	$$
	
	Now, let $f\colon{}\mathbf L_1\to\ \mathbf L_2$  and $g\colon{}\mathbf L_2\to\ \mathbf L_3$ 
	be  morphisms of \Vmodule{}s. 
	Then 
	$$
	\begin{array}{r@{\,\,}c@{\,\,}l}
		\mathbf J[\mathbf H, g\circ f](\alpha)&=&(g \circ f)\circ \alpha=g \circ (f\circ \alpha)=%
		g\circ \mathbf J[\mathbf H, f](\alpha)\\[0.2cm]
		&=&\mathbf J[\mathbf H, g](\mathbf J[\mathbf H, f](\alpha))=
		(\mathbf J[\mathbf H, g]\circ \mathbf J[\mathbf H, f])(\alpha).
	\end{array}$$
\end{proof} 

\begin{theorem}\label{funSFJrel}  Let $\mathbf H_1=(\mathbf A_1, F_1), %
	\mathbf H_2=(\mathbf A_2, F_2)$ be \VFsemilattice{}s, 
        $\mathbf L$ a \Vmodule{} 
	and $f\colon \mathbf H_1\to\ \mathbf H_2$ a lax morphism of\/  \VFsemilattice{}s. 
	Then there exists a homomorphism 
	$\mathbf J[f, \mathbf L]\colon \mathbf J[\mathbf H_2, \mathbf L]\to \mathbf J[\mathbf H_1, \mathbf L]$ 
	of\/ \Vframe{}s such that 
	$$(\mathbf J[f, \mathbf L](\alpha))(y)=\alpha(f(y))=(\alpha \circ f)(y)$$ 
	for all $\alpha\in T_{[\mathbf H_2, \mathbf L]}$ and  all $y\in G_1$. Moreover, $\mathbf J[-, \mathbf L]$ 
	is a  contravariant functor  from the category $\VFSUPL$ of 
 \VFsemilattice{}s to the category $\VSUPH$ of \Vframe{}s.
\end{theorem}

\begin{proof} We first show that, for every $\alpha\in T_{[\mathbf H_2, \mathbf L]}$, 
$(\mathbf J[f, \mathbf L](\alpha)\in T_{[\mathbf H_1, \mathbf L]}$.

Let $U\subseteq A_1$, $x\in A_1$ and $v\in V$. We compute: 
\begin{align*}
    (\mathbf J[f, \mathbf L](\alpha))(\bigvee U)&=\alpha(f(\bigvee U))=%
    \bigvee\{\alpha(f(u))\mid u\in U\}\\
    &= \bigvee\{  (\mathbf J[f, \mathbf L](\alpha))(u)\mid u\in U\}
\end{align*}
and 
\begin{align*}
   (\mathbf J[\mathbf H, f](\alpha))(v*x)&=\alpha(f(v*x))= %
   v* \alpha(f(x))=v* (\mathbf J[\mathbf H, f](\alpha))(x).
\end{align*}

Let's verify that $\mathbf J[f, \mathbf L]$ is a \Vframe{} homomorphism.
Let us take $\alpha, \beta \in \mathbf J[\mathbf H_2, \mathbf L]$. 
We have to check that 
$s(\alpha,\beta) \leq r(\mathbf J[f, \mathbf L](\alpha), \mathbf J[f, \mathbf L](\beta))$.

We know that 
\begin{align*}
    s(\alpha,\beta) &= \bigwedge_{x \in A_2} \{\beta(x) \rightarrow \alpha(F_2(x))\}\\
    r(\mathbf J[f, \mathbf L](\alpha), \mathbf J[f, \mathbf L](\beta))&= %
    \bigwedge_{y \in A_1} \{\mathbf J[f, \mathbf L](\beta)(y) \rightarrow \mathbf J[f, \mathbf L](\alpha)(F_1(y))\}
\end{align*}

Let $y\in G_1$. Then $f(y)\in G_2$. We compute:
\begin{align*}
s(\alpha,\beta)\otimes\mathbf J[f, \mathbf L](\beta))(y)&=%
s(\alpha,\beta)\otimes \beta(f(y))\leq  \alpha(F_2(f(y)))\leq \alpha(f(F_1(y)))\\
&=\mathbf J[f, \mathbf L](\alpha))(F_1(y)).
\end{align*}
We conclude $s(\alpha,\beta)\leq \mathbf J[f, \mathbf L](\beta)(y) \rightarrow \mathbf J[f, \mathbf L](\alpha)(F_1(y))$ for all  $y\in G_1$, i.e., 
$s(\alpha,\beta)\leq r(\mathbf J[f, \mathbf L](\alpha), \mathbf J[f, \mathbf L](\beta))$.

Let us check that  $\mathbf J[-, \mathbf L]$ is a functor. 
Let $\mathbf H=(\mathbf A, F)$ be a \VFsemilattice, 
$\mbox{\rm id}_{A}$ the identity morphism of \VFsemilattice{}s on $A$ and 
	$\alpha\in T_{[\mathbf H, \mathbf L]}$. We compute:
	$$
	\mathbf J[\mbox{\rm id}_{A}, \mathbf L](\alpha)=%
	\alpha \circ \mbox{\rm id}_{A}=\alpha=\mbox{\rm id}_{\mathbf J[\mathbf H, \mathbf L]}(\alpha). 
	$$

Let us now prove the composition condition:

Now, let $f\colon{}\mathbf H_1\to\ \mathbf H_2$  and $g\colon{}\mathbf H_2\to\ \mathbf H_3$ 
	be   lax morphisms of \VFsemilattice{}s. We have:
	$$
	\begin{array}{r@{\,\,}c@{\,\,}l}
		\mathbf J[g\circ f,\mathbf L](\alpha)&=&\alpha\circ (g \circ f)=%
		(\alpha\circ g) \circ f=  \mathbf J[g,\mathbf L](\alpha) \circ f\\[0.2cm]
		&=&\mathbf J[f,\mathbf L](\mathbf J[g,\mathbf L](\alpha))=%
		(\mathbf J[f,\mathbf L]\circ \mathbf J[g,\mathbf L])(\alpha).
	\end{array}$$
\end{proof}

\section{Three Adjoint Situations}\label{Three Adjoint Situations}

In this section, we build upon the findings of \cite[Section 4]{Bot} to introduce three adjoint situations. These adjunctions are specifically designed to accommodate the unique properties inherent in quantale-valued structures. They are given by:
\begin{enumerate}
	\item $(\eta, \varepsilon)\colon (\mathbf{J}\otimes -)\dashv (-^{\mathbf{J}})\colon \VSUPH \to \VFSUPL$,
	\item $(\varphi, \psi)\colon (-\otimes \mathbf{H})\dashv (\mathbf{J}[\mathbf{H},-])\colon \VSUPH \to \VJ$, and
	\item $(\nu, \mu)\colon (\mathbf{J}[-,\mathbf{L}])\dashv (\mathbf{L}^{-})\colon \VJ \to {\VFSUPL}^{op}$.
\end{enumerate}

We now briefly explain the intuitive meaning behind each of the three adjoint situations.

\textbf{Adjunction $(\eta, \varepsilon)$:}
The functor $\mathbf{J}\otimes -$ takes a \Vmodule{} with a fuzzy tense operator (an object in $\VFSUPL$) and, using a fixed \Vframe{} $\mathbf{J}$, produces a standard \Vmodule{} (an object in $\VSUPH$). Intuitively, this process adds a temporal aspect, governed by $\mathbf{J}$, to the standard \Vmodule{}. Conversely, the right adjoint $(-^{\mathbf{J}})$ associates with a \Vmodule{} the space of $\mathbf{V}$--valued functions over $\mathbf{J}$, representing generalized fuzzy characteristic functions. This adjoint pair gives a natural isomorphism between
\[\mathrm{Hom}_{\VFSUPL}(\mathbf{J}\otimes \mathbf{L}, \mathbf{M}) \cong \mathrm{Hom}_{\VSUPH}(\mathbf{L}, \mathbf{M}^{\mathbf{J}}),\]
thereby extending the classical connection between the $\mathbf{J}\otimes -$ construction and its function space counterpart from \cite{Bot}.

\textbf{Adjunction $(\varphi, \psi)$:}
In this case, the functor $-\otimes \mathbf{H}$ acts on an object from $\VJ$, forming its tensor product with a fixed \VFsemilattice{} $\mathbf{H}$. This operation effectively "fuzzifies" the structure of a \Vmodule{} by incorporating the tense operator from $\mathbf{H}$. The right adjoint, given by the \Vframe{} construction $\mathbf{J}[\mathbf{H},-]$, recovers a \Vframe{} by gathering suitable $\mathbf{V}$--module homomorphisms that preserve the fuzzy tense structure. This yields a natural bijection between
\[\mathrm{Hom}_{\VSUPH}(\mathbf{L}\otimes \mathbf{H}, \mathbf{M}) \cong \mathrm{Hom}_{\VJ}(\mathbf{L}, \mathbf{J}[\mathbf{H},\mathbf{M}]),\]
which formalizes the interaction between the $-\otimes \mathbf{H}$ construction and the extraction of a frame structure within the quantale-valued context.

\textbf{Adjunction $(\nu, \mu)$:}
The final adjunction formalizes a duality between the construction of \Vframe{}s and the recovery of \Vmodule{} structures. Here, the functor $\mathbf{J}[-,\mathbf{L}]$ associates with each object in $\VFSUPL$ a \Vframe{} constructed from the \Vmodule{} homomorphisms into a fixed \Vmodule{} $\mathbf{L}$. Its left adjoint, denoted by $\mathbf{L}^{-}$, reverses this process by evaluating the frame to reconstruct the original \Vmodule{}. This adjunction is expressed through a natural bijection between
\[
\mathrm{Hom}_{\VFSUPL}(\mathbf{L}^{\mathbf{J}}, \mathbf{H}) \text{ and } \mathrm{Hom}_{\VJ}(\mathbf{J}, \mathbf{J}[\mathbf{H}, \mathbf{L}]),
\]
thereby revealing a dual correspondence between the formation of frames from fuzzy tense operators and the evaluation of these frames to retrieve the corresponding quantale-valued structure.

\begin{theorem}[Generalized version of Theorem 4.1. from \cite{Bot}]
Let $\mathbf J=(T,r)$ be a \Vframe{}. Then:
\begin{enumerate}
\item For an arbitrary \VFsemilattice{} $\mathbf H=(\mathbf A, F)$ there exists a lax morphism 
$\eta _{\mathbf H}\colon \mathbf H\to (\mathbf J\otimes \mathbf H)^{\mathbf J}$ 
of \VFsemilattice{}s defined in such a way that 
$$(\eta _{\mathbf H}(x))(i)=\mbox{\rm n}(j[\mathbf J,\mathbf H])(x_{i=}).$$ 
Moreover, $\eta=(\eta _{\mathbf H}\colon {\mathbf H}\to ({\mathbf J}\otimes {\mathbf H})^{\mathbf J})_{{\mathbf H}\in \indexVFSUPL}$ 
is a natural transformation.

\item For an arbitrary \Vmodule{} $\mathbf L$ there exists a unique morphism 
$\varepsilon_{\mathbf L} \colon \mathbf J\otimes {\mathbf L}^{\mathbf J}\to \mathbf L$ of 
\Vmodule{}s such that the following diagram commutes:
\end{enumerate}

\begin{center}
\begin{tikzpicture}[scale=0.7]

\node (alfa') at (12,0) {$\mathbf L$};
  
   \node (gama') at (0,5) {$(\mathbf L^T)^T$};
   \node (delta') at (12,5) {$\mathbf J\otimes \mathbf L^{\mathbf J}$};
   
   \node (1') at (6,6) {$\mbox{\rm n}(j[\mathbf J,\mathbf L^{\mathbf J}])$};
   \node (2') at (11,2.5) {$\varepsilon_{\mathbf L}$};
   \node (3') at (6,2) {$e_{\mathbf L}$};
 
  \draw [->](gama') -- (alfa');
  \draw [->](delta') -- (alfa');
  \draw [->](gama') -- (delta');

 \end{tikzpicture}
 \end{center}
where $e_{\mathbf L}\colon(\mathbf L^{T})^{T}\to \mathbf L$ is defined by 
$e_{\mathbf L}(\bar{x})=\bigvee _{i\in T} (\bar{x}(i))(i)$ for any $\bar{x}\in (L^{T})^{T}$.
Moreover, $\varepsilon$ is a natural transformation. It can be expressed as:
\begin{align*}
    \varepsilon=(\varepsilon_{\mathbf L} \colon \mathbf J\otimes {\mathbf L}^{\mathbf J}\to \mathbf L), 
\end{align*}
where ${\mathbf L}$ is an element of the category \VSUPH.
\item 3. There exists an adjoint situation 
$(\eta, \varepsilon)\colon (\mathbf  J\otimes -)\dashv (-^{\mathbf  J})\colon %
\VSUPH \to \VFSUPL$. 

\end{theorem}

\begin{proof}
{(a):} Let us consider an arbitrary \VFsemilattice{} $\mathbf H=(\mathbf A, F)$.  Evidently, 
$\eta _{\mathbf H}$ preserves arbitrary joins. 
Assume that $x\in A$ and $i\in J$.  We compute: 

\begin{align*}
(F^{\mathbf J} (\eta _{\mathbf H}(x)))(i)&%
=\bigvee \{r(i,k)*\eta_{\mathbf H}(x)(k)) \mid k\in T\}\\%
&%
=\bigvee \{r(i,k)*\mbox{\rm n}(j[\mathbf J,\mathbf H])(x_{k=}) \mid k\in T\}\\%
&%
=\bigvee \{\mbox{\rm n}(j[\mathbf J,\mathbf H])(r(i,k)*x_{k=}) \mid k\in T\}\\%
&%
=\mbox{\rm n}(j[\mathbf J,\mathbf H])\left(\bigvee \{r(i,k)*x_{k=} \mid k\in T\}\right) =
\mbox{\rm n}(j[\mathbf J,\mathbf H])(x_{ir})\\
&\leq\mbox{\rm n}(j[\mathbf J,\mathbf H])%
(x_{ir}\vee F(x)_{i=})%
=\mbox{\rm n}(j[\mathbf J,\mathbf H])(F(x)_{i=})\\  &=(\eta _{\mathbf H}(F(x)))(i).
\end{align*}

Therefore $F^{\mathbf J} \circ \eta _{\mathbf H}\leq \eta _{\mathbf H} \circ F$  and hence  
$\eta _{\mathbf H}$ is a lax morphism. Now, let us assume that 
$\mathbf H_1=(\mathbf A_1, F_1)$ and  %
$\mathbf H_2=(\mathbf A_2, F_2)$ are \VFsemilattice{}s, and that 
$f\colon \mathbf H_1\to \mathbf  H_2$ is a  lax morphism of 
\VFsemilattice{}s. 
We have to show that the following diagram commutes:
\begin{center}
\begin{tikzpicture}

\node (alfa') at (7,5) {$\mathbf H_1$};
  \node (beta') at (7,2) {$(\mathbf J\otimes \mathbf H_1)^{\mathbf J}$};
   \node (gama') at (12,2) {$(\mathbf T\otimes \mathbf H_2)^{\mathbf J}$};
   \node (delta') at (12,5) {$\mathbf H_2$};
   
   \node (1') at (9.5,4.5) {$f$};
   \node (2') at (9.5,2.5) {$(\mathbf J\otimes f)^{\mathbf J}$};
   \node (3') at (7.5,3.55) {$\eta _{\mathbf H_1}$};
   \node (4') at (11.5,3.55) {$\eta _{\mathbf H_2}$};
   
  \draw [->](alfa') -- (delta');
  \draw [->](beta') -- (gama');
  \draw [->](delta') -- (gama');
  \draw [->](alfa') -- (beta');
 \end{tikzpicture}
 \end{center}
 Assume that $x\in A_1$ and $i\in J$. 
 We compute (using first Proposition \ref{funSSF}, 
 then Proposition \ref{nucn} and again Proposition \ref{funSSF}): 
 
$$
\begin{array}{@{}r@{\,}c@{\,}l}
  (((\mathbf J\otimes f)^{\mathbf J}\circ\eta _{\mathbf H_1})(x))(i)&=&%
  (\mathbf J\otimes f)(\eta _{\mathbf H_1}(x))(i)=(\mathbf J\otimes f)%
  (\mbox{\rm n}(j[\mathbf J,\mathbf H_1])(x_{i=}))\\[0.2cm]
  &=& \mbox{\rm n}(j[\mathbf J,\mathbf H_2])(f^{\mathbf J}(x_{i=}))=%
  \mbox{\rm n}(j[\mathbf J,\mathbf H_2])(f(x)_{i=})\\[0.2cm]
  &=& ((\eta _{\mathbf H_2}\circ f)(x))(i).
  \end{array}
$$

(b): Let $\mathbf L$  be a  \Vmodule{}. It it is transparent that $e_{\mathbf L}$ preserves arbitrary joins. 

Assume that $\bar{x}\in (L^T)^T$,  $v\in V$ and $i, j\in T$.  We compute: 
\begin{align*}
(v*e_{\mathbf L})(\bar{x})= v*e_{\mathbf L}(\bar{x})=%
v*\bigvee _{i\in T} (\bar{x}(i))(i)=%
\bigvee _{i\in T} ((v*\bar{x})(i))(i)=%
e_{\mathbf L}(v*\bar{x}).
\end{align*}

Assume that $z\in L^T$  and $i\in T$.
 $$
 \begin{array}{@{}r@{\,}c@{\,}l}
 e_{\mathbf L}(z_{ir})&=&\bigvee _{k\in T} (z_{ir}(k))(k) =\bigvee _{k\in T} (r(i,k)*z)(k) =%
 \bigvee _{k\in T} r(i,k)*z(k)
 \\[0.2cm]
 &=&(F^{\mathbf J}(z))(i)=%
 \bigvee_{j\in T}\left(\left(F^{\mathbf J}(z)_{i=}\right)(j)\right)(j)=%
 e_{\mathbf L}(F^{\mathbf J}(z)_{i=}).
  \end{array}$$

By Lemma \ref{lemprenucleus2}   there is a unique morphism 
$\varepsilon_{\mathbf L}\colon \mathbf J\otimes \mathbf L^{\mathbf J}\to\ \mathbf L$
of \Vmodule{}s such that 
$e=\varepsilon_{\mathbf L}\circ %
\mbox{\rm n}(j[\mathbf J,{\mathbf L}^{\mathbf J}])$.  

Let us now consider a morphism $f\colon\mathbf L_1\to \mathbf L_2$ of \Vmodule{}s.  
We have to prove that the following diagram commutes:

\begin{center}
\begin{tikzpicture}

\node (alfa') at (7,5) {$\mathbf J\otimes \mathbf L_1^{\mathbf J}$};
  \node (beta') at (7,2) {$\mathbf L_1$};
   \node (gama') at (12,2) {$\mathbf L_2$};
   \node (delta') at (12,5) {$\mathbf J\otimes\mathbf  L_2^{\mathbf J}$};
   
   \node (1') at (9.5,4.5) {$\mathbf J\otimes f^{\mathbf J}$};
   \node (2') at (9.5,2.5) {$f$};
   \node (3') at (7.7,3.55) {$\varepsilon_{\mathbf L_1}$};
   \node (4') at (11.2,3.55) {$\varepsilon_{\mathbf L_2}$};
   
  \draw [->](alfa') -- (delta');
  \draw [->](beta') -- (gama');
  \draw [->](delta') -- (gama');
  \draw [->](alfa') -- (beta');
 \end{tikzpicture}
 \end{center}

Let $\overline{x}\in ({\mathbf L_1}^{T})^{T}$. We compute:
 $$
 \begin{array}{@{}r@{}c@{\,}l}
(\varepsilon_{\mathbf L_2}\circ  (\mathbf J\otimes f^{\mathbf J}))&(\mbox{\rm n}(j[\mathbf J,{\mathbf L_1}^{\mathbf J}])(\overline{x}))&=
\varepsilon _{\mathbf L_2}(\mbox{\rm n}(j[\mathbf J,{\mathbf L_2}^{\mathbf J}]({f^{\mathbf J}}^{\mathbf J}(\overline{x})))=
e_{\mathbf L_2}({f^{\mathbf J}}^{\mathbf J}(\overline{x}))\\[0.2cm]
&\multicolumn{2}{@{}l}{=\bigvee _{i\in T}(({f^{\mathbf J}}^{\mathbf J}(\overline{x}))(i))(i)=%
\bigvee _{i\in T}({f^{\mathbf J}}(\overline{x}(i)))(i)}\\[0.2cm]
&\multicolumn{2}{@{}l}{=\bigvee _{i\in T}{f}(\overline{x}(i)(i))=f(\bigvee _{i\in T}\overline{x}(i)(i))=%
f(e_{\mathbf L_1}(\overline{x}))}\\[0.2cm]
&\multicolumn{2}{@{}l}{=f((\varepsilon_{\mathbf L_1}\circ %
\mbox{\rm n}(j[\mathbf J,{\mathbf L_1}^{\mathbf J}])(\overline{x}))=%
(f\circ \varepsilon_{\mathbf L_1})%
(\mbox{\rm n}(j[\mathbf J,{\mathbf L_1}^{\mathbf J}])(\overline{x})).}
 \end{array}$$ 

(c): Let  $\mathbf H=(\mathbf A, F)$ be a \VFsemilattice{} 
and  $\mathbf L$   a  \Vmodule{}.
We will prove the commutativity of following diagrams:

\noindent
\begin{tabular}{@{}c c@{}c}
\begin{tikzpicture}[scale=0.35]

\node (alfa') at (12,0) {$\mathbf  J\otimes \mathbf  H$};
  
   \node (gama') at (0,5) {$\mathbf  J\otimes\mathbf  H$};
   \node (delta') at (12,5) {$\mathbf  J\otimes (\mathbf  J\otimes \mathbf  H)^{\mathbf J}$};
   
   \node (1') at (6,5.7) {$\mathbf  J\otimes \eta_{\mathbf  H}$};
   \node (2') at (13.7,2.7) {$\varepsilon_{\mathbf  J\otimes\mathbf  H}$};
   \node (3') at (6,1.5) {$\mbox{\rm id}_{\mathbf  J\otimes \mathbf  H}$};

  \draw [->](gama') -- (alfa');
  \draw [->](delta') -- (alfa');
  \draw [->](gama') -- (delta');
  
 \end{tikzpicture}
 &&
 \begin{tikzpicture}[scale=0.35]

\node (alfa') at (12,0) {${{\mathbf L}^{\mathbf J}}$};
  
   \node (gama') at (0,5) {${{\mathbf L}^{\mathbf J}}$};
   \node (delta') at (12,5) {$(\mathbf  J\otimes {{\mathbf L}^{\mathbf J}})^{\mathbf  J}$};
   
   \node (1') at (6,5.7) {$\eta _{\mathbf  L^{\mathbf  J}}$};
   \node (2') at (13.7,2.5) {$\varepsilon_{{\mathbf L}^{\mathbf J}}$};
   \node (3') at (6,1.5) {$\mbox{\rm id}_{{\mathbf L}^{\mathbf J}}$};

  \draw [->](gama') -- (alfa');
  \draw [->](delta') -- (alfa');
  \draw [->](gama') -- (delta');
  
 \end{tikzpicture}
\end{tabular}

For the first diagram assume $\overline{x}\in \mathbf A^T$.  From 
Proposition \ref{nucn}
we know that the following diagram commutes (when necessary we forget that some morphisms are actually 
from \VFSUPL{} and we work entirely in \VSUPH:

\begin{center}
\begin{tikzpicture}[scale=0.6]

\node (alfa') at (1,5) {$\mathbf A^{T}$};
  \node (beta') at (1,0) {$((\mathbf J\otimes \mathbf H)^{T})^{T}$};
   \node (gama') at (12,0) {$\mathbf J\otimes (\mathbf J\otimes \mathbf H)^{\mathbf J}$};
   \node (delta') at (12,5) {$\mathbf J\otimes \mathbf H$};
    \node (deltax') at (20,0) {$\mathbf J\otimes \mathbf H$};
      \node (deltaxz') at (20,-5) {$\mathbf J\otimes \mathbf H$};
   
   \node (1') at (7.5,5.5) {$\mbox{\rm n}(j[\mathbf J,\mathbf H])$};
   \node (2') at (6.5,0.5) {$\mbox{\rm n}(j[\mathbf J, (\mathbf J\otimes \mathbf H)^{\mathbf J}])$};
   \node (3') at (2,2.5) {$(\eta _{\mathbf H})^{\mathbf J}$};
   \node (4') at (13.7,2.5) {$\mathbf J\otimes \eta _{\mathbf H}$};
      \node (5') at (16.25,0.5) {$\varepsilon_{\mathbf  J\otimes\mathbf  H}$};
       \node (6') at (21.4,-2.5) {$\mbox{\rm id}_{\mathbf  J\otimes \mathbf  H}$};
        \node (7') at (14.35,-2.85) {$e_{\mathbf J\otimes {\mathbf H}}$};
         \node (8') at (17.5,2.6) {$\mbox{\rm id}_{\mathbf  J\otimes \mathbf  H}$};
   
  \draw [->](alfa') -- (delta');
  \draw [->](beta') -- (gama');
  \draw [->](delta') -- (gama');
  \draw [->](alfa') -- (beta');
    \draw [->](gama') -- (deltax');
     \draw [->](deltax') -- (deltaxz');
        \draw [->](beta') -- (deltaxz');
        \draw[densely dashed,->] (delta') --  (deltax');
 \end{tikzpicture}
 \end{center}

\vskip-0.5cm

We compute (using the same steps as in \cite[Theorem 4.1.]{Bot}):
 $$
 \begin{array}{@{}r@{}l@{}l}
\varepsilon_{\mathbf  J\otimes\mathbf  H}((\mathbf  J\otimes \eta_{\mathbf  H})&(\mbox{\rm n}(j[\mathbf J,{\mathbf A}^{\mathbf J}])(\overline{x})))&=%
(\varepsilon_{\mathbf  J\otimes\mathbf  H}\circ %
\mbox{\rm n}(j[\mathbf J, (\mathbf J\otimes \mathbf H)^{\mathbf J}]))%
((\eta _{\mathbf H})^{\mathbf J}(\overline{x}))\\[0.2cm]
&\multicolumn{2}{@{}l}{=%
e_{\mathbf  J\otimes\mathbf  H}((\eta _{\mathbf H})^{\mathbf J}(\overline{x}))%
=\bigvee_{i\in T} ((\eta _{\mathbf H})^{\mathbf J}(\overline{x})(i))(i)}\\[0.2cm]
&\multicolumn{2}{@{}l}{=%
\bigvee_{i\in T} (\eta _{\mathbf H}(\overline{x}(i)))(i)=%
\bigvee_{i\in T} \mbox{\rm n}(j[\mathbf J,\mathbf H])(\overline{x}(i)_{i=})} \\[0.2cm]
&\multicolumn{2}{@{}l}{=%
\mbox{\rm n}(j[\mathbf J,\mathbf H])(\bigvee_{i\in T} \overline{x}(i)_{i=}) =%
\mbox{\rm n}(j[\mathbf J,\mathbf H])(\overline{x})}.
  \end{array}$$ 

Hence $\varepsilon_{\mathbf  J\otimes\mathbf  H}\circ (\mathbf  J\otimes \eta_{\mathbf  H})=%
\mbox{\rm id}_{\mathbf  J\otimes \mathbf  H}$.

To show the commutativity of the second diagram assume that $\overline{x}\in {\mathbf L}^{T}$ 
and $i\in T$. We compute (again using the corresponding steps as in \cite[Theorem 4.1.]{Bot}): 
 $$
 \begin{array}{@{}r@{}l@{}l}
(((\varepsilon _{\mathbf  L})^{\mathbf  J} \circ  \eta_{{\mathbf L}^{\mathbf J}})&(\overline{x}))(i)&=%
((\varepsilon _{\mathbf  L})^{\mathbf  J}(\eta_{{\mathbf L}^{\mathbf J}}(\overline{x})))(i)=%
\varepsilon _{\mathbf  L}(\eta_{{\mathbf L}^{\mathbf J}}(\overline{x})(i))\\[0.2cm]%
&\multicolumn{2}{@{}l}{=%
\varepsilon _{\mathbf  L}(\mbox{\rm n}(j[\mathbf J,{\mathbf L}^{\mathbf J}])(\overline{x}_{i=}))
=e_{{\mathbf L}}(\overline{x}_{i=})=\bigvee_{k\in T} \overline{x}_{i=}(k)(k)=\overline{x}(i).}
  \end{array}$$ 
  
  We conclude that $(\varepsilon _{\mathbf  L})^{\mathbf  J} \circ  \eta_{{\mathbf L}^{\mathbf J}}=%
  \mbox{\rm id}_{{\mathbf  L}^{\mathbf  J}}$.

\end{proof}

\begin{theorem}  [Generalized version of Theorem 4.2. from \cite{Bot}] \label{G42}
 Let  $\mathbf H=(\mathbf A, F)$ be a \VFsemilattice{}. Then:
\begin{enumerate}
\item For an arbitrary \Vframe{} $\mathbf J=(T,r)$, there exists a unique homomorphism of \Vframe{}s 
$\varphi_{\mathbf J}\colon \mathbf J\to {\mathbf J}[\mathbf H, \mathbf J\otimes \mathbf H]$ defined for arbitrary $x\in A$ and $i\in T$ in such a way that 
$$(\varphi_{\mathbf J}(i))(x)=\mbox{\rm n}(j[\mathbf J,{\mathbf H}])(x_{i=}).$$ 
Moreover,  $\varphi=(\varphi_{\mathbf J}\colon %
{\mathbf J}\to {\mathbf J}[\mathbf H, \mathbf J\otimes \mathbf H])_{\mathbf J\in {\mathbb J}}$ 
is a natural transformation  between the identity functor on  $V$-$\mathbb J$ and the endofunctor  
${\mathbf J}[\mathbf H, -\otimes \mathbf H]$.

\item For an arbitrary \Vmodule{} $\mathbf L$ there exists a unique \Vmodule{} morphism 
$\psi_{\mathbf L} \colon {\mathbf J}[\mathbf H, \mathbf L]\otimes {\mathbf H}\to \mathbf L$ 
such that the following diagram commutes:

\begin{center}
\begin{tikzpicture}[scale=0.6]

\node (alfa') at (12,0) {$\mathbf L$};
  
   \node (gama') at (0,5) {${\mathbf G}^{{T}_{[\mathbf H, \mathbf L]}}$};
   \node (delta') at (12,5) {$\mathbf {\mathbf J}[\mathbf H, \mathbf L]\otimes \mathbf H$};
   
   \node (1') at (6,4.3) {$\mbox{\rm n}(j[{\mathbf J}[\mathbf H, \mathbf L],\mathbf H])$};
   \node (2') at (11,2.5) {$\psi_{\mathbf L}$};
   \node (3') at (6,2) {$f_{\mathbf L}$};

  \draw [->](gama') -- (alfa');
  \draw [->](delta') -- (alfa');
  \draw [->](gama') -- (delta');
  
 \end{tikzpicture}
 \end{center}
 
\noindent{}where $f_{\mathbf L}\colon {\mathbf G}^{{T}_{[\mathbf H, \mathbf L]}}\to \mathbf L$ 
 is defined by $f_{\mathbf L}({x})=\bigvee _{\alpha\in {{\mathbf J}[\mathbf H, \mathbf L]}} %
 \alpha({x}(\alpha))$ for any ${x}\in {G}^{{T}_{[\mathbf H, \mathbf L]}}$. Moreover, 
 $\psi=(\psi_{\mathbf L} \colon {\mathbf J}[\mathbf H, %
 \mathbf L]\otimes {\mathbf H}\to \mathbf L)_{\mathbf L\in \mathbb S}$ is a natural transformation 
 between the endofunctor ${\mathbf J}[\mathbf H,  -]\otimes {\mathbf H}$ and the identity functor on $V$-$\mathbb S$.

\item There exists an adjoint situation 
$(\varphi, \psi)\colon (-\otimes \mathbf  H)\dashv {\mathbf  J}[\mathbf  H,-])\colon %
$V$-\mathbb S \to $V$-\mathbb J$. 
\end{enumerate}

\end{theorem}

\begin{proof}
{(a):}  We have to show that our definition is correct. Assume $X\subseteq A$, $x\in A$,  $v\in V$ and $i\in I$. 
We have: 
$$
 \begin{array}{@{}r@{\,}c@{\,}l}
(\varphi_{\mathbf J}(i))\left(\bigvee X\right)&=&%
\mbox{\rm n}(j[\mathbf J,{\mathbf H}])\left(\left(\bigvee X\right)_{i=}\right)=%
\mbox{\rm n}(j[\mathbf J,{\mathbf H}])\left(\bigvee\{x_{i=}\mid x\in X\}\right)\\[0.2cm]
&=&%
\bigvee\{\mbox{\rm n}(j[\mathbf J,{\mathbf H}])\left(x_{i=}\right)\mid x\in X\}=%
\bigvee\{(\varphi_{\mathbf J}(i))\left(x\right)\mid x\in X\}
\end{array}
$$

and 

$$
 \begin{array}{@{}r@{\,}c@{\,}l}
(\varphi_{\mathbf J}(i))\left(v*x\right)&=&%
\mbox{\rm n}(j[\mathbf J,{\mathbf H}])\left(\left(v*x\right)_{i=}\right)=%
\mbox{\rm n}(j[\mathbf J,{\mathbf H}])\left(v* x_{i=}\right)\\[0.2cm]
&=&%
v*\mbox{\rm n}(j[\mathbf J,{\mathbf H}])\left(x_{i=}\right)=%
v*(\varphi_{\mathbf J}(i))\left(x\right).
\end{array}
$$

Hence $\varphi_{\mathbf J}(i)\in {T}_{[\mathbf H, \mathbf J\otimes \mathbf H]}$. 

Now let $i, j\in T$, $x\in A$. We compute:
\begin{align*}
r(i,j)*(\varphi_{\mathbf J} (j))(x)&= r(i,j)*\mbox{\rm n}(j[\mathbf J,{\mathbf H}])(x_{j=})%
= \mbox{\rm n}(j[\mathbf J,{\mathbf H}])(r(i,j)*x_{j=})\\
&\leq \mbox{\rm n}(j[\mathbf J,{\mathbf H}])(x_{ir})\leq %
\mbox{\rm n}(j[\mathbf J,{\mathbf H}])(x_{ir}\vee F(x)_{i=})\\
&=\mbox{\rm n}(j[\mathbf J,{\mathbf H}])(F(x)_{i=})=%
(\varphi_{\mathbf J}(i))(F(x)).
\end{align*}
Hence  $r(i,j) \leq \bigwedge_{x \in A} %
\big(\varphi_{\mathbf J}(j)(x)\rightarrow \varphi_{\mathbf J}(i)(F(x))\big)= r'(\varphi_{\mathbf J}(i),\varphi_{\mathbf J}(j))$
and $\varphi_{\mathbf J}$ is a \Vframe{} homomorphism.

Recall that the endofunctor 
${\mathbf J}[\mathbf H,  -]\otimes {\mathbf H}$ on $V$-$\mathbb J$ is a composition of functors $-\otimes {\mathbf H}$ and 
 ${\mathbf J}[\mathbf H,  -]$. 

Now, let  $t\colon{}\mathbf J_1\to \mathbf J_2$ be a homomorphism of %
\Vframe{}s. We have to show that the following diagram commutes:
\begin{center}
\begin{tikzpicture}
\node (alfa') at (7,5) {$\mathbf J_1$};
  \node (beta') at (7,2) {${\mathbf J}[\mathbf  H,\mathbf J_1\otimes \mathbf H]$};
   \node (gama') at (13,2) {${\mathbf J}[\mathbf H, \mathbf J_2\otimes \mathbf H]$};
   \node (delta') at (13,5) {$\mathbf J_2$};
   
   \node (1') at (9.5,4.5) {$t$};
   \node (2') at (10.05,2.5) {${\mathbf J}[\mathbf H,t\otimes \mathbf H]$};
   \node (3') at (7.5483,3.55) {$\varphi_{\mathbf J_1}$};
   \node (4') at (12.42,3.55) {$\varphi_{\mathbf J_2}$};
   
  \draw [->](alfa') -- (delta');
  \draw [->](beta') -- (gama');
  \draw [->](delta') -- (gama');
  \draw [->](alfa') -- (beta');
 \end{tikzpicture}
\end{center}

Let $i\in T_1$, $k\in T_2$ and $x\in A$. We obtain 
\begin{align*}
(t^{\rightarrow}(x_{i=}))(k)=\bigvee\{x_{i=}(j)\mid t(j)=k\}=x_{t(i)=}(k).
\end{align*}
Hence $t^{\rightarrow}(x_{i=})=x_{t(i)=}$. We compute: 
$$
 \begin{array}{@{}r@{}l@{}l}
(({\mathbf J}[\mathbf H,t\otimes \mathbf H]\circ \varphi_{\mathbf J_1})&(i))(x)&%
=\left({\mathbf J}[\mathbf  H, t\otimes \mathbf H](\varphi_{\mathbf J_1}(i))\right)(x)
\\[0.2cm]
&\multicolumn{2}{@{}l}{=\left((t\otimes \mathbf H)\circ (\varphi_{\mathbf J_1}(i))\right)(x)%
=\left(t\otimes \mathbf H\right)\left((\varphi_{\mathbf J_1}(i))(x)\right)}\\[0.2cm]
&\multicolumn{2}{@{}l}{=%
\left(t\otimes \mathbf H\right)\left(\mbox{\rm n}(j[\mathbf J_1,{\mathbf H}])(x_{i=})\right)%
=\mbox{\rm n}(j[\mathbf J_2,{\mathbf H}])(x_{t(i)=})} \\[0.2cm]
&\multicolumn{2}{@{}l}{=\left(\varphi_{\mathbf J_2}(t(i))\right)(x)
=\left((\varphi_{\mathbf J_2}\circ t)(i)\right)(x).}
\end{array}$$

{(b):}   It is transparent that $f_{\mathbf L}$ preserves arbitrary joins. 

Let us denote the relation $S_{[\mathbf H,\mathbf L]}$ as $S^{\bullet}$ and let $x\in G$, $\alpha \in T_{[\mathbf H,\mathbf L]}$ be arbitrary. We compute:
$$
\begin{array}{@{}r@{\,}l@{\,}l}
f_{\mathbf L}(x_{\alpha S_{[\mathbf H,\mathbf L]}})&=&%
\bigvee \{ \beta(x_{\alpha S_{[\mathbf H,\mathbf L]}}(\beta)) \mid {\beta\in T_{[\mathbf H,\mathbf L]}}\}\\[0.2cm]%
&=&\bigvee \{\beta (x)\mid \alpha \mathrel{S_{[\mathbf H,\mathbf L]}} \beta, {\beta\in T_{[\mathbf H,\mathbf L]}}\}%
\leq \alpha (F(x))\\[0.2cm]%
&=&\bigvee \{\beta(F(x)_{\alpha=}(\beta)) \mid {\beta\in T_{[\mathbf H,\mathbf L]}}\}=%
f_{\mathbf L}(F(x)_{\alpha=}).
\end{array}
$$ 

Therefore $f_{\mathbf L}(x_{\alpha S^{\bullet}}\vee F(x)_{\alpha=})=f_{\mathbf L}(F(x)_{\alpha=})$, which assures 
by  Lemma \ref{lemprenucleus2} the existence of $\psi_{\mathbf L}$ from the theorem. 

Let  $g\colon \mathbf L_1\to \mathbf L_2$  be a morphism of \Vmodule{}s. Let us 
show that the following diagram commutes:
\begin{center}
\begin{tikzpicture}[scale=0.8]
\node (alfa') at (7,5) {${\mathbf J}[\mathbf H, \mathbf L_1]\otimes \mathbf H$};
  \node (beta') at (7,0) {$\mathbf L_1$};
   \node (gama') at (14,0) {$\mathbf L_2$};
   \node (delta') at (14,5) {${\mathbf J}[\mathbf H,\mathbf L_2]\otimes \mathbf H$};
   \node (epsilon') at (3,8) {$\mathbf G^{T_{[\mathbf H,\mathbf L_1]}}$};
   \node (mu') at (18,8) {$\mathbf G^{T_{[\mathbf H,\mathbf L_2]}}$};
   
   \node (1') at (10.5,5.5) {${\mathbf J}[\mathbf H,g]\otimes \mathbf H$};
   \node (2') at (10.5,0.5) {$g$};
   \node (3') at (7.6,2.5) {$\psi_{{\mathbf L}_1}$};
   \node (4') at (13.3,2.5) {$\psi_{\mathbf L_2}$};
    \node (5') at (7.098595,6.6510595) {$\mbox{\rm n}(j[\mathbf J[\mathbf H, \mathbf L_1],{\mathbf H}])$};
       \node (6') at (5.09876,2.5) {$f_{{\mathbf L}_1}$};
    \node (7') at (13.8595,6.6510595) {$\mbox{\rm n}(j[\mathbf J[\mathbf H, \mathbf L_2],{\mathbf H}])$};
           \node (8') at (15.876509876,2.5) {$f_{{\mathbf L}_2}$};
        \node (9') at (10.5,8.5) {${\mathbf J}[\mathbf H,g]^{\rightarrow}$};      
   
  \draw [->](alfa') -- (delta');
  \draw [->](beta') -- (gama');
  \draw [->](delta') -- (gama');
  \draw [->](alfa') -- (beta');
   \draw [->](epsilon') -- (alfa');
      \draw [->](epsilon') -- (beta');
       \draw [->](mu') -- (delta');
   \draw [->](mu') -- (gama');
    \draw [->](epsilon') -- (mu');
 \end{tikzpicture}
 \end{center}

Let $x\in G^{T_{[\mathbf H,\mathbf L_1]}}$. We compute: 
$$
\begin{array}{@{}r@{\,}c@{\,}l}
(\psi_{\mathbf L_2}&\circ& ({\mathbf J}[\mathbf H,g]\otimes \mathbf H)\circ %
\mbox{\rm n}(j[\mathbf J[\mathbf H, \mathbf L_1],{\mathbf H}]))(x)\\[0.2cm]
&=&  %
(\psi_{\mathbf L_2}\circ \mbox{\rm n}(j[\mathbf J[\mathbf H, \mathbf L_2],{\mathbf H}])\circ %
{\mathbf J}[\mathbf H,g]^{\rightarrow})(x) =%
(f_{{\mathbf L}_2}\circ {\mathbf J}[\mathbf H,g]^{\rightarrow})(x)\\[0.2cm]
&=&%
f_{{\mathbf L}_2}((\bigvee\{x(\alpha) \mid g\circ \alpha=\beta, %
\alpha \in T_{[\mathbf H,\mathbf L_1]}\})_{\beta \in T_{[\mathbf H,\mathbf L_2]}})\\[0.2cm]
&=&\bigvee _{\beta \in T_{[\mathbf H,\mathbf L_2]}} %
 \beta((\bigvee\{x(\alpha) \mid g\circ \alpha=\beta, %
\alpha \in T_{[\mathbf H,\mathbf L_1]}\}))\\[0.2cm]
&=&
 \bigvee\{g(\alpha(x(\alpha))) \mid  %
\alpha \in T_{[\mathbf H,\mathbf L_1]}\}=
g(\bigvee\{\alpha(x(\alpha)) \mid  \alpha \in T_{[\mathbf H,\mathbf L_1]}\})\\[0.2cm]
&=&(g\circ f_{{\mathbf L}_1})(x)=(g\circ %
\psi_{\mathbf L_1}\circ \mbox{\rm n}(j[\mathbf J[\mathbf H, \mathbf L_1],{\mathbf H}])(x).
\end{array}
$$ 
Since $\mbox{\rm n}(j[\mathbf J[\mathbf H, \mathbf L_1],{\mathbf H}])$ is surjective 
we have 
$\psi_{\mathbf L_2}\circ ({\mathbf J}[\mathbf H,g]\otimes \mathbf H)=%
g\circ \psi_{\mathbf L_1}$.

(c): Let  $\mathbf J=(T, r)$ be a \Vframe{} and  $\mathbf L$   a  \Vmodule{}.
We will prove the commutativity of following diagrams:

\noindent{}%
\resizebox{\textwidth}{!}{%
\begin{tabular}{@{}c c@{}c}
\begin{tikzpicture}[scale=0.4]

\node (alfa') at (12,0) {$\mathbf  J\otimes \mathbf  H$};
  
   \node (gama') at (0,5) {$\mathbf  J\otimes\mathbf  H$};
   \node (delta') at (12,5) {${\mathbf J}[\mathbf H, \mathbf  J\otimes\mathbf  H]\otimes \mathbf H$};
   
   \node (1') at (5.5,5.7) {$\varphi_{\mathbf  J}\otimes {\mathbf  H}$};
   \node (2') at (13.7,2.7) {$\psi_{\mathbf  J\otimes\mathbf  H}$};
   \node (3') at (6,1.5) {$\mbox{\rm id}_{\mathbf  J\otimes \mathbf  H}$};

  \draw [->](gama') -- (alfa');
  \draw [->](delta') -- (alfa');
  \draw [->](gama') -- (delta');
  
 \end{tikzpicture}
 &&
 \begin{tikzpicture}[scale=0.4]

\node (alfa') at (12,0) {${\mathbf J}[\mathbf H, {\mathbf L}]$};
  
   \node (gama') at (0,5) {${\mathbf J}[\mathbf H, {\mathbf L}]$};
   \node (delta') at (12,5) {${\mathbf J}[\mathbf H, {\mathbf J}[\mathbf H, \mathbf  L]\otimes \mathbf H]$};
   
   \node (1') at (5.6,5.7) {$\varphi_{{\mathbf J}[\mathbf H, {\mathbf L}]}$};
   \node (2') at (14.1099,2.5) {${{\mathbf J}[\mathbf H, \psi_{\mathbf L}]}{}$};
   \node (3') at (6,1.5) {$\mbox{\rm id}_{{\mathbf J}[\mathbf H, {\mathbf L}]}$};

  \draw [->](gama') -- (alfa');
  \draw [->](delta') -- (alfa');
  \draw [->](gama') -- (delta');
  
 \end{tikzpicture}
\end{tabular}%
}

Let $\overline{x}\in  G^T$.  According to  ([3],Theorem 3.6),
we know that the following diagram commutes:

\begin{center}
\begin{tikzpicture}[scale=0.73]
\node (alfa') at (6,5) {$\mathbf G^{T}$};
  \node (beta') at (6,0) {$\mathbf G^{T_{[\mathbf H, \mathbf J\otimes \mathbf H]}}$};
   \node (gama') at (15,0) {${\mathbf J}[\mathbf H, \mathbf J\otimes \mathbf H]\otimes \mathbf H$};
   \node (delta') at (15,5) {$\mathbf J\otimes \mathbf H$};
    \node (deltax') at (22,0) {$\mathbf J\otimes \mathbf H$};
      \node (deltaxz') at (22,-3.5) {$\mathbf J\otimes \mathbf H$};
   
   \node (1') at (10.5,5.5) {$\mbox{\rm n}(j[\mathbf J,\mathbf H])$};
   \node (2') at (10,0.5) {$\mbox{\rm n}(j[{\mathbf J}[\mathbf H, \mathbf J\otimes \mathbf H], %
   \mathbf H])$};
   \node (3') at (6.85,2.5) {$\varphi _{\mathbf J}^{\rightarrow}$};
   \node (4') at (13.9,2.5) {$\varphi _{\mathbf J}\otimes\mathbf  H$};
      \node (5') at (18.5,0.5) {$\psi _{\mathbf J\otimes\mathbf  H}$};
       \node (6') at (20.9,-1.5) {$\mbox{\rm id}_{\mathbf  J\otimes \mathbf  H}$};
        \node (7') at (13.59,-2.85) {$f_{\mathbf J\otimes {\mathbf H}}$};
         \node (8') at (19.9,2.5) {$\mbox{\rm id}_{\mathbf  J\otimes \mathbf  H}$};
   
  \draw [->](alfa') -- (delta');
  \draw [->](beta') -- (gama');
  \draw [->](delta') -- (gama');
  \draw [->](alfa') -- (beta');
    \draw [->](gama') -- (deltax');
     \draw [->](deltax') -- (deltaxz');
        \draw [->](beta') -- (deltaxz');
        \draw[densely dashed,->] (delta') --  (deltax');
 \end{tikzpicture}
 \end{center}
\vskip-0.3cm

We compute:

$$
\begin{array}{@{}r@{\,}c@{\,}l}
(&\psi_{\mathbf J\otimes\mathbf  H}&\circ (\varphi _{\mathbf J}\otimes\mathbf  H)\circ %
\mbox{\rm n}(j[\mathbf J,\mathbf H]))(x)=
(\psi _{\mathbf J\otimes\mathbf  H}\circ%
\mbox{\rm n}(j[{\mathbf J}[\mathbf H, \mathbf J\otimes \mathbf H], %
   \mathbf H])\circ %
\varphi _{\mathbf J}^{\rightarrow})(x)\\[0.2cm]
&\multicolumn{2}{@{}l}{=(f_{\mathbf J\otimes {\mathbf H}}\circ %
\varphi _{\mathbf J}^{\rightarrow})(x)=%
f_{\mathbf J\otimes {\mathbf H}}%
\left(\bigvee\{x(i)\mid \varphi_{\mathbf J}(i)=\alpha\})_{\alpha\in T_{[\mathbf H, %
\mathbf J\otimes \mathbf H]}}\right)}\\[0.2cm]
&\multicolumn{2}{@{}l}{=\bigvee_{\alpha\in T_{[\mathbf H, %
\mathbf J\otimes \mathbf H]}}%
\alpha(\bigvee\{x(i)\mid \varphi_{\mathbf J}(i)=\alpha\})%
=\bigvee_{\alpha\in T_{[\mathbf H, %
\mathbf J\otimes \mathbf H]}}%
\bigvee_{\varphi_{\mathbf J}(i)=\alpha, i\in T}\alpha(x(i))}\\[0.2cm]
&\multicolumn{2}{@{}l}{=\bigvee\{(\varphi_{\mathbf J}(i))(x(i)) \mid 
{i\in T}\}=%
\bigvee\{\mbox{\rm n}(j[\mathbf J,{\mathbf H}])(x(i)_{i=}) \mid 
{ i\in T}\}}\\[0.2cm]
&\multicolumn{2}{@{}l}{=\mbox{\rm n}(j[\mathbf J,{\mathbf H}])(\bigvee\{ x(i)_{i=} \mid  i\in T\})=%
\mbox{\rm n}(j[\mathbf J,{\mathbf H}])(x).}
\end{array}
$$

Hence the first diagram commutes. Now, let  $\alpha \in T_{[\mathbf H,\mathbf L]}$ and  $x\in G$. 
We obtain: 
$$
\begin{array}{@{}r@{\,}l@{\,}l}
(({{\mathbf J}[\mathbf H, \psi_{\mathbf L}]}{} &\circ \varphi_{{\mathbf J}[\mathbf H, {\mathbf L}]})(\alpha))(x)&=%
({{\mathbf J}[\mathbf H, \psi_{\mathbf L}]}{}(\varphi_{{\mathbf J}[\mathbf H, {\mathbf L}]}(\alpha))(x) %
\\[0.2cm]
&\multicolumn{2}{@{}l}{%
=\psi_{\mathbf L}((\varphi_{{\mathbf J}[\mathbf H, {\mathbf L}]}(\alpha))(x))=\psi_{\mathbf L}(\mbox{\rm n}(j[\mathbf J[\mathbf H, {\mathbf L}],{\mathbf H}])(x_{\alpha=}))=%
f_{\mathbf L}(x_{\alpha=})}\\[0.2cm]
&\multicolumn{2}{@{}l}{=\bigvee\{\beta(x_{\alpha=}(\beta)) \mid \beta\in T_{[\mathbf H, {\mathbf L}]}\}=\alpha (x)}
\end{array}$$
which yields the commutativity of the second diagram.
\end{proof}

\begin{theorem} [Generalized version of Theorem 4.3. from \cite{Bot}] 
Let $\mathbf L$ be a \Vmodule{}. Then the following holds:
\begin{enumerate}
\item For an arbitrary \Vframe{} $\mathbf J=(T,r)$, there exists a unique homomorphism of \Vframe{}s 
$\nu_{\mathbf J}\colon \mathbf J\to {\mathbf J}[{\mathbf L}^{\mathbf J}, \mathbf L]$ defined for 
arbitrary $x\in L^T$ and $i\in T$ in such a way that 
$$(\nu_{\mathbf J}(i))(x)=x({i}).$$ 

Moreover,  $\nu=(\nu_{\mathbf J}\colon %
{\mathbf J}\to {\mathbf J}[{\mathbf L}^{\mathbf J}, \mathbf L])_{\mathbf J\in {\mathbb V-J}}$ 
is a natural transformation.\\

\item  For an arbitrary $V$-$F$-sup-semilattice $\mathbf H=(\mathbf G, F)$ there exists a lax morphism 
$\mu_{\mathbf H}\colon {\mathbf H}\to %
\mathbf L^{\mathbf J[\mathbf H, \mathbf L]}$ 
of $V$-$F$-sup-semilattices defined for arbitrary $x\in G$ and $\alpha\in T_{\mathbf J[\mathbf H, \mathbf L]}$ 
by  
$$(\mu_{\mathbf H}(x))(\alpha)=\alpha(x).$$ 

Moreover, 
$\mu=(\mu_{\mathbf H}\colon {\mathbf H}\to %
\mathbf L^{\mathbf J[\mathbf H, \mathbf L]})_{{\mathbf H}\in V-\smFSUPL}$ 
is a natural transformation.\\
\item There exists an adjoint situation 
$(\nu, \mu)\colon {\mathbf  J}[-, \mathbf  L])\dashv {\mathbf  L}^{-}\colon %
\mathbb V-J \to V-{\FSUPL}^{op}$. 
\end{enumerate}
\end{theorem}

\begin{proof}
{(a):} Let $i, j\in T$, $x\in L^{T}$. We compute:
\begin{align*}
r'(\nu_{\mathbf J}(i),\nu_{\mathbf J}(j))&=%
\bigwedge_{x \in L^T} \{\nu_{\mathbf J} (j(x))\rightarrow (\nu_{\mathbf J}(i))(F^J(x))\}\\
&=\bigwedge_{x \in L^T} \{x(j)\rightarrow (F^{\mathbf J}(x)(i))\}.
\end{align*}
Since $(F^{\mathbf J}(x))(i)=\bigvee \{r(i,j)*x(j) \mid {j\in T}\}$ we 
conclude that $(F^{\mathbf J}(x))(i)\geq r(i,j)*x(j)$ for all $j\in T$.
Therefore, 
\begin{align*}
r(i,j) \leq \bigwedge_{x \in L^T} \{\nu_{\mathbf J} (j(x))\rightarrow (\nu_{\mathbf J}(i))(F^J(x))\}=r'(\nu_{\mathbf J}(i),\nu_{\mathbf J}(j)).
\end{align*}

Hence $\nu_{\mathbf J}$ is a \Vframe{} homomorphism.

Assume now that $t\colon \mathbf J_1 \to \mathbf J_2$ is a \Vframe{} homomorphism between 
\Vframe{}s $\mathbf J_1=(T_1,r_1)$ and $\mathbf J_2=(T_2,s_2)$. We have to show 
that the following diagram commutes:
\begin{center}
\begin{tikzpicture}

\node (alfa') at (7,5) {$\mathbf J_1$};
  \node (beta') at (7,2) {${\mathbf J}[{\mathbf L}^{\mathbf J_1}, \mathbf L]$};
   \node (gama') at (13,2) {${\mathbf J}[{\mathbf L}^{\mathbf J_2}, \mathbf L]$};
   \node (delta') at (13,5) {$\mathbf J_2$};
   
   \node (1') at (9.5,4.5) {$t$};
   \node (2') at (10.05,2.5) {${\mathbf J}[{\mathbf L}^{t},  \mathbf L]$};
   \node (3') at (7.5483,3.55) {$\nu_{\mathbf J_1}$};
   \node (4') at (12.42,3.55) {$\nu_{\mathbf J_2}$};
   
  \draw [->](alfa') -- (delta');
  \draw [->](beta') -- (gama');
  \draw [->](delta') -- (gama');
  \draw [->](alfa') -- (beta');
 \end{tikzpicture}
\end{center}

Assume that $i\in T_1$ and $x\in  L^{T_2}$. We compute:
$$
\begin{array}{@{}r@{\,}c@{\,}l}
\left(({\mathbf J}[{\mathbf L}^{t},  \mathbf L]\circ \nu_{\mathbf J_1})(i)\right)(x)&=&%
({\mathbf J}[{\mathbf L}^{t},  \mathbf L]( \nu_{\mathbf J_1}(i))(x)=%
( \nu_{\mathbf J_1}(i)\circ {\mathbf L}^{t})(x)=%
\nu_{\mathbf J_1}(i)({\mathbf L}^{t}(x))\\[0.2cm]
&=&%
\nu_{\mathbf J_1}(i)(x\circ t)=(x\circ t)(i)=x(t(i))=%
 (\nu_{\mathbf J_2}(t(i))(x)\\[0.2cm]
&=&%
\left((\nu_{\mathbf J_2}\circ t)(i)\right)(x).\\
\end{array}
$$
Hence ${\mathbf J}[{\mathbf L}^{t},  \mathbf L]\circ \nu_{\mathbf J_1}=%
\nu_{\mathbf J_2}\circ t$.

{(b):}  Evidently, $\mu_{\mathbf H}$ preserves arbitrary joins. We have to verify that 
$F^{{\mathbf J[\mathbf H, \mathbf L]}}\circ \mu_{\mathbf H}\leq \mu_{\mathbf H}\circ F$. 
Let $x\in G$ and $\alpha\in T_{\mathbf J[\mathbf H, \mathbf L]}\}$. We compute: 
$$
\begin{array}{@{}r@{\,}c@{\,}l}
\left((F^{{\mathbf J[\mathbf H, \mathbf L]}}\circ \mu_{\mathbf H})(x)\right)(\alpha)&=&
(F^{{\mathbf J[\mathbf H, \mathbf L]}}(\mu_{\mathbf H}(x)))(\alpha)\\[0.2cm]
&=&%
\bigvee\{(r(\alpha,\beta)*(\mu_{\mathbf H}(x))(\beta) \mid \beta\in T_{\mathbf J[\mathbf H, \mathbf L]}\} \\[0.2cm]
&=&%
\bigvee\{(r(\alpha,\beta)*\beta(x) \mid \beta\in T_{\mathbf J[\mathbf H, \mathbf L]}\} \\[0.2cm]
&\leq&%
\alpha(F(x))=\mu_{\mathbf H}(F(x))(\alpha)=%
\left((\mu_{\mathbf H}\circ F) (x)\right)(\alpha).
\end{array}
$$

Therefore is $\mu_{\mathbf H}$ a lax morphism. 

Now, let us assume that 
$\mathbf H_1=(\mathbf G_1, F_1)$ and  %
$\mathbf H_2=(\mathbf G_2, F_2)$ are $V$-$F$-sup-semilattices, and that 
$f\colon \mathbf H_1\to \mathbf  H_2$ is a  lax morphism of $V$-$F$-sup-semilattices. 
We have to verify that the following diagram commutes:
\begin{center}
\begin{tikzpicture}

\node (alfa') at (7,5) {$\mathbf H_1$};
  \node (beta') at (7,2) {$\mathbf L^{\mathbf J[\mathbf H_1, \mathbf L]}$};
   \node (gama') at (12,2) {$\mathbf L^{\mathbf J[\mathbf H_2, \mathbf L]}$};
   \node (delta') at (12,5) {$\mathbf H_2$};
   
   \node (1') at (9.5,4.5) {$f$};
   \node (2') at (9.5,2.5) {$\mathbf L^{\mathbf J[f, \mathbf L]}$};
   \node (3') at (7.5,3.55) {$\mu_{\mathbf H_1}$};
   \node (4') at (11.5,3.55) {$\mu_{\mathbf H_2}$};
   
  \draw [->](alfa') -- (delta');
  \draw [->](beta') -- (gama');
  \draw [->](delta') -- (gama');
  \draw [->](alfa') -- (beta');
 \end{tikzpicture}
 \end{center}
 Assume that $x\in G_1$ and $\alpha\in T_{[\mathbf H_2, \mathbf L]}$. We compute: 
 
 $$
\begin{array}{@{}r@{}l@{}l}
\left((\mathbf L^{\mathbf J[f, \mathbf L]} \right.&\left.\circ \mu_{\mathbf H_1})(x)\right)(\alpha)&=%
\left(\mathbf L^{\mathbf J[f, \mathbf L]} (\mu_{\mathbf H_1}(x))\right)(\alpha)\\[0.2cm]
&\multicolumn{2}{@{}l}{=%
\left(\mu_{\mathbf H_1}(x) \right)({\mathbf J[f, \mathbf L](\alpha)})=%
({\mathbf J[f, \mathbf L](\alpha)})(x)%
=(\alpha\circ f)(x)}\\[0.2cm]
&\multicolumn{2}{@{}l}{=%
\alpha(f(x))=\left((\mu_{\mathbf H_2}\circ f)(x)\right)(\alpha).}%
\end{array}
$$
hence $\mathbf L^{\mathbf J[f, \mathbf L]} \circ \mu_{\mathbf H_1}=\mu_{\mathbf H_2}\circ f$.

 (c): Let  $\mathbf J=(T, r)$ be a \Vframe{} and  $\mathbf H=(\mathbf G, F)$ a $V$-$F$-sup-semilattice.
We will prove the commutativity of following diagrams:

\begin{center}
\begin{tabular}{@{}c c c}
\begin{tikzpicture}[scale=0.35]

\node (alfa') at (12,0) {$\mathbf  L^{\mathbf  J}$};
  
   \node (gama') at (0,5) {$\mathbf  L^{\mathbf  J}$};
   \node (delta') at (12,5) {${\mathbf L}^{{\mathbf J}%
   [\mathbf L^{\mathbf J}, \mathbf  L]}$};
   
   \node (1') at (5.5,5.7) {$\mu_{{\mathbf  L}^{\mathbf  J}}$};
   \node (2') at (13.7,2.7) {${\mathbf  L}^{\nu_{\mathbf  J}}$};
   \node (3') at (6,1.5) {$\mbox{\rm id}_{\mathbf  L^{\mathbf  J}}$};

  \draw [->](gama') -- (alfa');
  \draw [->](delta') -- (alfa');
  \draw [->](gama') -- (delta');
  
 \end{tikzpicture}
 &\phantom{xx}&
 \begin{tikzpicture}[scale=0.35]

\node (alfa') at (12,0) {${\mathbf J}[\mathbf H, {\mathbf L}]$};
  
   \node (gama') at (0,5) {${\mathbf J}[\mathbf H, {\mathbf L}]$};
   \node (delta') at (12,5) {${\mathbf J}[\mathbf L^{{\mathbf J}[\mathbf H, \mathbf  L]}, \mathbf L]$};
   
   \node (1') at (5.6,5.7) {$\nu_{{\mathbf J}[\mathbf H, {\mathbf L}]}$};
   \node (2') at (14.3099,2.5) {${{\mathbf J}[\mu_{\mathbf H}, {\mathbf L}]}{}$};
   \node (3') at (5.2469,1.5) {$\mbox{\rm id}_{{\mathbf J}[\mathbf H, {\mathbf L}]}$};

  \draw [->](gama') -- (alfa');
  \draw [->](delta') -- (alfa');
  \draw [->](gama') -- (delta');
  
 \end{tikzpicture}
\end{tabular}
\end{center}

Let $x\in L^{T}$ and $i\in T$. We compute: 
 $$
\begin{array}{@{}r@{\,}l@{\,}l}
\left(({\mathbf  L}^{\nu_{\mathbf  J}} \circ \mu_{{\mathbf  L}^{\mathbf  J}})(x)\right)(i)&=&%
\left({\mathbf  L}^{\nu_{\mathbf  J}} (\mu_{{\mathbf  L}^{\mathbf  J}}(x))\right)(i)=%
(\mu_{{\mathbf  L}^{\mathbf  J}}(x))\left({\nu_{\mathbf  J}} (i)\right)=%
\left({\nu_{\mathbf  J}} (i)\right)(x)\\[0.2cm]
&=&%
x(i)=(\mbox{\rm id}_{\mathbf  L^{\mathbf  J}}(x))(i).
\end{array}
$$
Hence ${\mathbf  L}^{\nu_{\mathbf  J}} \circ \mu_{{\mathbf  L}^{\mathbf  J}}=%
\mbox{\rm id}_{\mathbf  L^{\mathbf  J}}$ in $\FSUPL$. Assume that $x\in G$ and $\alpha\in T_{[\mathbf H, \mathbf L]}$. We compute: 
 $$
\begin{array}{@{}r@{\,}c@{\,}l}
\left(({{\mathbf J}[\mu_{\mathbf H}, {\mathbf L}]}{}\right. &\left.\circ\, \, %
\nu_{{\mathbf J}[\mathbf H, {\mathbf L}]})(\alpha)\right)(x)&=%
\left({{\mathbf J}[\mu_{\mathbf H}, {\mathbf L}]}{} (\nu_{{\mathbf J}[\mathbf H, {\mathbf L}]}(\alpha))\right)(x)
\\[0.2cm]
&\multicolumn{2}{@{}l}{=%
(\nu_{{\mathbf J}[\mathbf H, {\mathbf L}]}(\alpha))\left(\mu_{\mathbf H}(x)\right)%
=\left(\mu_{\mathbf H}(x)\right)(\alpha)=%
\alpha(x)=(\mbox{\rm id}_{{\mathbf J}[{\mathbf H}, {\mathbf L}]}(\alpha))(x).}
\end{array}
$$
Therefore ${{\mathbf J}[\mu_{\mathbf H}, {\mathbf L}]}{} \circ %
\nu_{{\mathbf J}[\mathbf H, {\mathbf L}]}=\mbox{\rm id}_{{\mathbf J}[{\mathbf H}, {\mathbf L}]}$. 
\end{proof}

	\section{A prototypical example}\label{Examples}

The purpose of this section is to illustrate the 
adjoint pair $(\nu, \mu)\colon {\mathbf  J}[-, \mathbf  L])$ through a prototypical example.

It is well known that the diamond lattice 
$\mathbf M_3=(\{0, a, b, c, 1\}, \bigvee, \otimes, b)$ is a non-distributive non-commutative unital quantale (see \cite[Nr. 5.2.13]{Eklund})
if it is equipped with the following operations $\bigvee$ and $\otimes$ 
(and the induced operation $\wedge$):
$$
\begin{array}{c c c c c}

\begin{array}{c | c c c c c}
	\otimes& 0 &a & b & c & 1\\ \hline 
    0&	0&	0&	0&	0&	0\\
    a&	0&	0&	a&	a&	a\\
    b&	0&	a&	b&	c&	1\\
    c&	0&	a&	1&	1&	1\\
    1&	0&	a&	1&	1&	1\\
\end{array}& &
\begin{array}{c | c c c c c}
	\bigvee& 0 &a & b & c & 1\\ \hline 
	0&	0&	a & b & c & 1\\
	a&	a&	a&	1&	1&	1\\
	b&	b&	1&	b&	1&	1\\
	c&	c&	1&	1&	c&	1\\
	1&	1&	1&	1&	1&	1\\
\end{array}	& &
\begin{array}{c | c c c c c}
\wedge& 0 &a & b & c & 1\\ \hline 
0&	0&	0 & 0 & 0 & 0\\
a&	0&	a&	0&	0&	a\\
b&	0&	0&	b&	0&	b\\
c&	0&	0&	0&	c&	c\\
1&	0&	a&	b&	c&	1\\
\end{array}	
\end{array}
$$

The chain $\mathbf V=(\{0, b, 1\}, \bigvee_{\mathbf V}, \otimes_{\mathbf V}, b)$ is 
a commutative distributive unital subquantale of $\mathbf M_3$  and it is equipped with the following operations 
$\bigvee_{\mathbf V}$ and $\otimes_{\mathbf V}$ 
(and the induced operation $\wedge_{\mathbf V}$):

$$
\begin{array}{c c c c c}
	
	\begin{array}{c | c c c}
		\otimes_{\mathbf V}& 0 &b & 1\\ \hline 
		0&	0&		0&		0\\
		b&	0&		b&		1\\
		1&	0&		1&		1\\
	\end{array}& &
	\begin{array}{c | c c c }
		\bigvee_{\mathbf V}& 0 & b  & 1\\ \hline 
		0&	0 & b &  1\\
		b&	b&		b&		1\\
		1&	1&		1&		1\\
	\end{array}	& &
	\begin{array}{c | c c c}
		\wedge_{\mathbf V}& 0 & b & 1\\ \hline 
		0&	0&	 0 &  0\\
		b&	0&		b&		b\\
		1&	0&		b&		1\\
	\end{array}	
\end{array}
$$

Both $\mathbf M_3$ and $\mathbf V$ are depicted in  Figure \ref{M3V}.
Evidently, since $\mathbf M_3$ is a quantale we have that 
$\mathbf A=(\{0, a, b, c, 1\}, \bigvee, *_{\mathbf A})$ is a \Vmodule{}, 
here $*_{\mathbf A}$ is a restriction of $\otimes$ to the set 
$V\times M_3$.

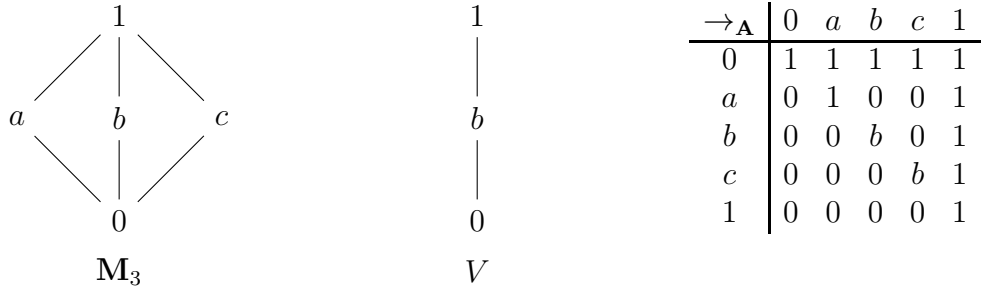
\begin{figure}[h]
\begin{center}
\begin{tikzpicture}[scale=0.68]
		\node (1_2) at (7,6) {$1$};
		\node (b2) at (9,4) {$c$};
		\node (a2) at (5,4) {$a$};
		\node (c2) at (7,4) {$b$};
		\node (0_2) at (7,2) {$0$};
		\node (10_2) at (7,1.0) {${\mathbf M}_3$};
		\draw (1_2) -- (b2) -- (0_2) -- (a2) -- (1_2) -- (c2) -- (0_2);

\node (1_1) at (14,6) {$1$};
		\node (0_1) at (14,2) {$0$};
\node (0_1/2) at (14,4) {$b$};
		\node (100_2) at (14,1.0) {$V$};

\draw (0_1/2) -- (1_1);
		\draw (0_1) -- (0_1/2);

  \node (T_1) at (21,4) {$\begin{array}{c | c c c c c}
\rightarrow_{\mathbf A}&%
    0 & a & b & c & 1\\ \hline 
0&	1&	1 & 1 & 1 & 1\\
a&	0&	1&	0&	0&	1\\
b&	0&	0&	b&	0&	1\\
c&	0&	0&	0&	b&	1\\
1&	0&	0&	0&	0&	1\\
\end{array}	$};
	\end{tikzpicture}
 
 \end{center}
\caption{ $\mathbf M_3$, its subquantale $\mathbf V$ and 
$\rightarrow_{\mathbf A}$ table}\label{M3V}
\end{figure}

We now put $F(x)=a*x$ for all $x\in A$. Then $F$ preserves arbitrary joins and 
$F(u*x)=a*(u*x)=(a*u)*x=(u*a)*x=u*(a*x)=u*F(x)$ for all $u\in \{0,b, c\}$ and 
$x\in A$. 
$$
	\begin{array}{c | c c c c c}
		x& 0 &a & b & c & 1\\ \hline 
		F(x)&	0&	0&	a&	a&	a\\		
	\end{array}
$$

Let ${\mathbf L}=(\{0, 1\}, \bigvee, *_{\mathbf L}\}$ be a \Vsubmodule{} of $\mathbf A$. 
The action $*_{\mathbf L}\colon V \times L \longrightarrow L$ on $L$ and its 
respective adjoint $\rightarrow_{\mathbf L}\colon L \times L \longrightarrow V$ are given by:
$$\begin{array}{c  c c}
\begin{array}{c | c c}
    *_{\mathbf L} & 0 & 1 \\
    \hline
    0 &   0 & 0  \\
    \hline
    b &   0 & 1  \\
    \hline
    1 &   0 & 1  
\end{array}&\phantom{xxxxx} &\begin{array}{c | c c}
    \rightarrow_{\mathbf L} & 0 & 1 \\
    \hline
    0 &   1 & 1  \\
    \hline
    1 &   0 & 1  
\end{array}
\end{array}
$$

Let us define a \Vframe{} ${{\mathbf J}{[{\mathbf A},{\mathbf L}]}}%
=({\mathbb S}({\mathbf A},{\mathbf L}), %
r)$ where $r$ 
is a map defined in Section 3 as $r(\alpha, \beta)=\bigwedge_{x \in G} \{\beta(x) \rightarrow \alpha(F(x))\}$. Clearly, ${\mathbb S}({\mathbf G},{\mathbf L})$ potentially has 8 elements (see [3], Example 5.1), which we will denote $f_i$, where $i\in \{1,2,3,4,5,6,7,8\}$ and their description is given by Table \ref{functfi}.

Every one of these potential morphism also has to satisfy 
$f(v \ast_{\mathbf A}  y) = v \ast_{\mathbf L}  f(y)$ for every $v \in V$ and every $y \in G$. If $v\in \{0,b\}$ we always have that 
$f(v \ast_{\mathbf A}  y) = v \ast_{\mathbf L}  f(y)$. Hence it remains to check whether 
$f(1 \ast_{\mathbf A}  y) = 1 \ast_{\mathbf L}  f(y)=f(y)$. We put 
$\overline{f}_i= f_i\circ (1 \ast_{\mathbf A} -)$.

\begin{table}[h!]
\centering
{\begin{tabular}{l | cccccccc | c |c@{\,\,}c@{\,\,}c}
& $f_1$ &  $f_2$&  $f_3$&  $f_4$&$f_5$ & %
$f_6$ & $f_7$ &$f_8$ &$1 \ast_{\mathbf A}  -$& $f_1\circ F$&%
$f_7\circ F$ & $f_8\circ F$\\[0.01cm] \hline
&  &  &  &  &  &  &  &  &   &  &\\[-0.351cm]
$0$ & 0 & 0 & 0 & 0 & 0 & 0 & 0 & 0&0& 0 & 0&0\\
$a$ & 0 & 0 & 0 & 1 & 1 & 1 & 0 & 1&$a$& 0 & 0&0\\
$b$ & 0 &0  & 1 & 0 &1  &0  & 1 &1 &1& 0 & 0&1\\
$c$ & 0 &1  & 0 & 0 & 0 & 1 & 1 &1 &1& 0 & 0&1\\
  $1$    & 0 & 1 & 1 & 1 & 1 & 1 &1  &1 &1& 0 & 0&1\\[0.01cm] \hline
  &  &  &  &  &  &  &  &  &   &  &\\[-0.351cm]
  & $\overline{f}_1$ &  $\overline{f}_2$&  $\overline{f}_3$&
  $\overline{f}_4$&$\overline{f}_5$ & %
  $\overline{f}_6$ & $\overline{f}_7$ &$\overline{f}_8$ & $F$  &  &\\[0.01cm] \hline
&  &  &  &  &  &  &  &  &   &  &\\[-0.351cm]
$0$ & 0 & 0 & 0 & 0 & 0 & 0 & 0 & 0& 0 &  &\\
$a$ & 0 & 0 & 0 & 1 & 1 & 1 & 0 & 1& 0 &  &\\
$b$ &  0 & \framebox{1} & 1 & \framebox{1} & 1 & \framebox{1} &1  &1 & $a$ &  &\\
$c$  & 0 & 1 & \framebox{1} & \framebox{1} & \framebox{1} & 1 &1  &1 & $a$ &  &\\
  $1$    & 0 & 1 & 1 & 1 & 1 & 1 &1  &1 & $a$ &  &
\end{tabular}}
\caption{Table of possible elements of ${{\mathbf J}{[{\mathbf A},{\mathbf L}]}}$}\label{functfi}
 \end{table}

From  Table \ref{functfi} we see that  $f_1$, $f_7$ and $f_8$ are the only morphisms that actually satisfy this property. For any other $i$ we have marked by framebox a contradiction.

The map $r$ is given by the following table:
$$
\begin{array}{c | c c c}
    r & f_1 & f_7 & f_8 \\
    \hline
    f_1 &1   &0   &0   \\
    \hline
    f_7 &1   &0   &0   \\
    \hline
    f_8 &1   &1   &0   
\end{array}
$$

Namely, for all $i \in\{1,7,8\}$ it holds that
$r(f_i, f_1)=\bigwedge_{x \in A}  \{f_1(x) \rightarrow f_i(F(x))\}=1$
since $f_1(x)=0$  for all $x \in A$.

Since $f_7(1)=f_8(1)=1$ and $f_1(F(1))=f_7(F(1))=0$ we have that 
$r(f_1,f_7)=r(f_1,f_8)=r(f_7,f_7)=r(f_7,f_8)=0$. 

Since $f_8(a)=1$ and $f_8(F(a))=0$ we have that 
$r(f_8,f_8)=0$. 

Since $f_7=f_8\circ F$ we have that $r(f_8,f_7)=1$.

By the previous, there exists a lax morphism 
$\mu _{\mathbf H}\colon{}{\mathbf H} \longrightarrow %
{\mathbf L}^{{\mathbf J}{[{\mathbf H},{\mathbf L}]}}$ of $V$-$F$-sup-semilattices 
defined for arbitrary $x\in A$ and $f\in {\mathbb S}({\mathbf A},{\mathbf L})$ by
$$(\mu _{\mathbf H}(x))(f)=f(x).$$

Let us now compute $\mu _H$ on elements of $A$. It holds that:
$$(\mu _{\mathbf H}(x))(f_1)=f_1(x)=0$$
for all $x \in A$, 
and $(\mu _{\mathbf H}(x))(f_8)=f_8(x)=0$ if $x=0$ and $(\mu _{\mathbf H}(x))(f_8)=f_8(x)=1$ otherwise, and $(\mu _{\mathbf H}(x))(f_7)=f_7(x)=0$ if $x\in \{0,a\}$ and $(\mu _{\mathbf H}(x))(f_7)=f_7(x)=1$ otherwise.

$$
\begin{array}{c | c c c}
     & f_1 & f_7 & f_8 \\
    \hline
    \mu _{\mathbf H}(0) &0   &0   &0   \\
    \hline
    \mu _{\mathbf H}(a) &0   &0   &1   \\
    \hline
    \mu _{\mathbf H}(b) &0   &1  &1   \\
    \hline
    \mu _{\mathbf H}(c) &0   &1   &1   \\
    \hline
    \mu _{\mathbf H}(1) &0   &1   &1   
\end{array}
$$
We see that the lax morphism $\mu _{\mathbf H}$ is not injective and so it is not an embedding.

\section{Conclusion}\label{conclusion}

In this paper, we have introduced three fundamental construction techniques:
\begin{enumerate}
\item Creating an {\VFsemilattice} from a \Vmodule\  and a {\Vframe},
\item Deriving a \Vmodule\  from a {\VFsemilattice} and a {\Vframe},
\item Generating a {\Vframe} from an {\VFsemilattice} and 
a \Vmodule{}.
\end{enumerate}

Through these constructions, we have established three corresponding adjoint situations between the relevant categories. This framework provides a unified perspective on recent theoretical developments regarding tense operator representations across various categories.

\section*{Acknowledgments.} The first author acknowledges support from the Czech Science Foundation (GAČR) project 23-09731L ``Representations of algebraic semantics for substructural logics''. The second and third author were supported by the Austrian Science Fund (FWF) [10.55776/PIN5424624] and the Czech Science Foundation (GAČR) project 25-20013L ``Orthogonality and Symmetry''. The third author acknowledges support from the Masaryk University project MUNI/A/1457/2023.

\end{document}